\documentclass[10pt]{amsart}
\usepackage{amssymb,amsmath,amsthm,amscd,mathrsfs}

\newtheorem{teo}{Theorem}[section]
\newtheorem{pro}[teo]{Proposition}
\newtheorem{lem}[teo]{Lemma}
\newtheorem{cor}[teo]{Corollary}

\theoremstyle{definition}
\newtheorem{dfn}[teo]{Definition}
\newtheorem{exa}[teo]{Example}

\theoremstyle{remark}
\newtheorem{rem}[teo]{Remark}

\begin{document}

\title[Singularities of Brill-Noether loci]{Singularities of Brill-Noether loci for vector bundles on a curve}
\author[Casalaina-Martin]{Sebastian Casalaina-Martin }
\thanks{The first author was partially supported by NSF MSPRF grant DMS-0503228.  He would like to thank Harvard University for its hospitality during the preparation of this paper.}
\address{University of Colorado at Boulder, Department of Mathematics,  Campus Box 395, Boulder, CO 80303, USA}
\email{casa@math.colorado.edu}

\author[Teixidor i Bigas]{Montserrat Teixidor i Bigas}
\address{Mathematics Department, Tufts University, Medford MA 02155, USA}
\email{montserrat.teixidoribigas@tufts.edu}

\date{\today}

\begin{abstract}
In this paper we consider the singularities of the varieties
parameterizing stable vector bundles of fixed rank and degree with
sections on a smooth curve of genus at least two. In particular, we
extend results of Y. Laszlo, and of the second author, regarding the
singularities of generalized theta divisors.
\end{abstract}
\maketitle

\section*{Introduction}

Given a smooth projective curve $C$ over $\mathbb{C}$ of genus $g\ge 2$, let $\mathscr{U}(r,d)$  denote the moduli
space of stable vector bundles of rank $r$ and degree $d$ on $C$.  It is well known that this is a smooth
 quasi-projective variety of dimension $r^2(g-1)+1$, and for non-negative integers $k$ there are subschemes
  $B^k_{r,d}\subseteq \mathscr U(r,d)$ parameterizing those vector bundles whose space of global sections is at
   least $k$ dimensional; i.e. as a set
$$
B^k_{r,d}=\{[E]\in \mathscr{U}(r,d): h^0(C,E)\ge k\}.
$$

The  $B^k_{r,d}$ can  be viewed as degeneracy loci, and as a
consequence the ``expected'' dimension of $B^k_{r,d}$ is  given by
the function
$$
\rho(g,k,\mathscr U(r,d)):= \dim \mathscr U(r,d) -k(-\chi +k),
$$
where $\chi=d+r(1-g)$.  It is still not fully understood when these subschemes have the expected dimension,
 or even when they are nonempty.  We refer the reader to
\cite{gt} for more details on this question.

The goal of this paper is to study the singularities of these schemes.   In particular, given a point
 $[E]\in B^k_{r,d}$, we would like to relate the multiplicity of $B^k_{r,d}$ at $[E]$ to the space of global sections
$H^0(C,E)$.  In the case of line bundles of degree $g-1$, this
question goes back to Riemann: the Riemann Singularity Theorem
states that for $[L]\in
\mathscr{U}(1,g-1)=\textnormal{Pic}^{g-1}(C)$,
$\textnormal{mult}_{[L]} B^1_{1,g-1}=h^0(L)$. This was generalized
by Kempf \cite{kt}, who showed that for all $0\le d\le g-1$, and
$[L]\in \mathscr{U}(1,d)=\textnormal{Pic}^{d}(C)$,
$\textnormal{mult}_{[L]} B^1_{1,d}=\binom {h^1(L)}{h^0(L)-1}$. In
another direction, Riemann's result was generalized in \cite{l} by
Laszlo, where it was  shown that if $d=r(g-1)$,  and $[E]\in
B^1_{r,r(g-1)}$, then $\textnormal{mult}_{[E]} B^1_{r,r(g-1)}
=h^0(C,E)$. We point out that the Zariski tangent spaces to the
$B^k_{r,d}$  are well understood:  at a point $E\in B^k_{r,d}$,  if
$E\notin B^{k+1}_{r,d}$, the tangent space to $B^k_{r,d}$
 at $E$ is
 $$
 T_EB^k_{r,d}=(\operatorname{Image} \mu)^\perp
 $$
 where
$$
\mu:H^0(E)\otimes H^0(K_C\otimes E^\vee)\to H^0(E\otimes E^\vee \otimes K_C)
$$
is the cup product mapping, and if $E\in  B^{k+1}_{r,d}$, then
$T_EB^k_{r,d}=T_E\mathscr U(r,d)$ (see \cite{gt} for more details).
 Using work of Kempf \cite{kt} as motivation, the following extends the results above by giving a description
  of $C_EB^k_{r,d}$, the tangent cone to  $B^k_{r,d}$ at $E$.

\begin{teo}\label{teo1}
  Suppose $0\le \rho(g,1,\mathscr{U}(r,d))< \dim \mathscr{U}(r,d)$.  For all  $E\in B^1_{r,d}$,
  the tangent cone $C_E B_{k,d}$ is a reduced, Cohen-Macaulay and normal subvariety of
   $T_E\mathscr U(r,d)=H^1(E^\vee\otimes E)$, defined by the $h^0(E)\times h^0(E)$ minors of an
   $h^0(E)\times h^1(E)$ matrix with entries in $H^1(E^\vee \otimes E)^\vee$.
Moreover, $ \textnormal{mult}_E B^1_{r,d}=\binom{h^1(E)}{h^0(E)-1}$.
    \end{teo}

The case $r=1$ is a special case of Kempf's theorem \cite{kt}.  Theorem \ref{kempfteo} gives more details
 about the matrix defining the tangent cone, analogous to the stronger statement in Kempf's paper, and the extension
  to the case $k>1$ in  Arbarello et al. \cite[Theorem VI 2.1]{acgh}.  The proofs of Theorems \ref{teo1} and
  \ref{kempfteo} are essentially the same as the proofs given in \cite{kt} and \cite{acgh};
   we give an outline in Section \ref{seckempf}.

There are certain generalizations of Clifford's theorem to rank two and rank three vector bundles, due to
 Cilleruelo-Sols \cite{cs} and Lange-Newstead \cite{ln}, respectively, which together with Theorem \ref{teo1}
 give bounds on the multiplicity of singular points of the $B^1_{r,d}$ for $r=2,3$.  Results of this type have
 been shown before.  For instance,
using a  ``Clifford'' type theorem for rank two vector bundles of degree $2g-2$, coupled with the special case
of Theorem \ref{teo1} cited above, Laszlo  showed that for $E\in \mathscr U(2,2(g-1))$, if $C$ is not hyperelliptic,
 then $\textnormal{mult}_E B^1_{2,2(g-1)}\le g$  (Remark after \cite{l} Proposition IV.2).
It would be interesting to know if one could use multiplier ideals to give similar bounds.
  In Section \ref{multid} we make a short remark along these lines using results of Ein-Lazarsfeld
  \cite{el} on pluri-theta divisors on principally polarized abelian varieties.
   The bounds  we obtain in this way are far from those noted above.

 Given a vector bundle $E\in \mathscr U(r',d')$, one can also consider the subvarieties
 $B^k_E \subseteq \mathscr U(r,d)$ parameterizing those vector bundles which, when tensored by $E$,
  have at least $k$ linearly independent global sections; i.e. as a set
$$
B^k_E=\{M\in \mathscr U(r,d):h^0(E\otimes M)\ge k\}.
$$
  Although many of the arguments used to prove Theorem \ref{teo1} carry over to this situation as well,
   it is difficult to verify when certain cup product computations needed for the proof can actually be carried out.
In Section 2, we state some partial results in these cases similar to Theorem \ref{teo1},
 where we assume that these cup products have the appropriate properties (see Remark \ref{remgencup}).

In the case that $k=1$ and $rd'+r'd=rr'(g-1)$, the expected
codimension of $B^1_E$ is one, and in this case we will use the
notation $\Theta_E=B^1_E$ for these ``generalized theta divisors.''
When $\Theta_E$ is a divisor, i.e. when $\Theta_E\ne \mathscr
U(r,d)$, it is a standard result that for $M\in \mathscr U(r,d)$,
$\textnormal{mult}_M\Theta_E\ge h^0(E\otimes M)$, and the first goal
will be to determine when equality holds. In section 3.1, we provide
a proof of the following:

\begin{teo} \label{teo2}
Suppose $rd'+r'd=rr'(g-1)$.
Let $E\in \mathscr U(r',d')$ and suppose $\Theta_E \subseteq \mathscr U(r,d)$ is a divisor.
 Given $M\in \mathscr U(r,d)$ and an extension of vector bundles
$$
0\rightarrow L \rightarrow E \rightarrow F \rightarrow 0,
$$
it follows that
$$\textnormal{mult}_M\Theta_E\ge h^0(L\otimes M)+h^1(F\otimes M).$$
Moreover, if $h^0(E\otimes M)\le 2$, then
$\textnormal{mult}_M\Theta_E>h^0(E\otimes M)$ if and only if
there exists an extension as above such that
$h^0(L\otimes M)+h^1(F\otimes M)>h^0(E\otimes M)$.\end{teo}

This generalizes a result of Laszlo \cite[V.7]{l} in  the case that $r=1$, $d=0$, $r'=2$ and $d'=2(g-1)$,
 i.e. $E\in \mathscr U(2,2(g-1))$ and $M\in \textnormal{Pic}^0(C)$.
 There is an example in that paper, given in Proposition IV.9  (cf. Remark \ref{remex}),
 showing that when $h^0(E\otimes M)>2$ it is possible  to have $\textnormal{mult}_M\Theta_E>h^0(E\otimes M)$
 even if such an extension does not exist.  On the other hand, Theorem \ref{teo2} gives a complete description
 of the singular locus of $\Theta_E$:

\begin{cor}
$M\in \textnormal{Sing } \Theta_E$ if and only if  either
\begin{enumerate}
\item  $h^0(E\otimes M)\ge 2$, or
\item $h^0(E\otimes M)=1$, and there exists an extension of vector bundles
$$
0\rightarrow L \rightarrow E \rightarrow F \rightarrow 0,
$$
such that $h^0(L\otimes M)+h^1(F\otimes M)>h^0(E\otimes M)$.

\end{enumerate}
\end{cor}

It should be pointed out that in the case that $r=1$,  $d=0$ and $d'=2(g-1)$, i.e. $E\in \mathscr U(2,2(g-1))$
 and $M\in \textnormal{Pic}^0(C)$, a result of the second author \cite{t} shows that if $E$ is general
 in $\mathscr U(2,2(g-1))$, then case (2) above does not arise.  We extend this to the case where $E$ has arbitrary
 rank in Section \ref{secexc}.  On the other hand, in Theorem \ref{teoexc} we generalize a result of Laszlo
 \cite[Corollary V.5]{l}  to show that for any curve, there are stable vector bundles $E$ of all ranks greater than one
  where case (2) does occur.

The results cited so far use  first order deformation theory. By looking at higher order deformation theory,
in some special cases, we are able to determine the exact multiplicity at points $M$ such that
 $\textnormal{mult}_M \Theta_E>h^0(E\otimes M)$. In \S \ref{sectd}, we
 prove the following:

\begin{teo}\label{pro}
Suppose that $L$ and $K_C\otimes F^\vee$ are general line bundles in
$B^1_d(C)$, and that $e\in H^1(F^\vee \otimes L)$ is general.  Let
$$
0\rightarrow L \rightarrow E \rightarrow F\rightarrow 0
$$
be the extension associated to $e$. Then $E\in \mathscr
U(2,2(g-1))$, and
$$\textnormal{mult}_{\mathscr O_C} \Theta_E=h^0(L)+h^1(F)>h^0(E).$$
\end{teo}

Roughly speaking, this says that if $E\in \mathscr U(2,2(g-1))$ and
is ``general''
 among the vector bundles for which there exists an extension
$$
0\rightarrow L \rightarrow E\otimes \xi\rightarrow F\rightarrow 0
$$
with $h^0(L)+h^1(F)>h^0(E\otimes \xi)$, then
$\textnormal{mult}_\xi \Theta_E=h^0(L)+h^1(F)$.   The notion of ``general'' is discussed in more detail in Section 3.

When $\Theta_E\subseteq \textnormal{Pic}^0(C)$, the tangent cones to the singularities of the theta divisor
 can be related to the geometry of the canonical model of the curve.  To be precise, in the case that $r=1$, $d=0$,
  and $d'=r'(g-1)$, so that $\Theta_E\subseteq \textnormal{Pic}^0(C)$, it follows that
for $\xi\in \textnormal{Pic}^0(C)$,
the tangent cone $\mathbb{P}C_\xi \Theta_E$ is contained in
$\mathbb{P}H^0(C,K_C)^\vee$.
Thus the tangent cone, the canonical model of the curve and the secant varieties to the canonical model are all
 contained in the same space.
In the case that $E$ is a line bundle of degree $g-1$, there is the standard result
(see Arbarello et al. \cite[Theorem VI 1.6 (i)]{acgh})
 showing that the $n$-secant variety, i.e. the union of the $n$-dimensional planes meeting  the canonical model
 in $n+1$ points, is contained in
$\mathbb{P}C_\xi \Theta_E$ if and only if $h^0(E\otimes \xi)\ge n+2$.
In Section \ref{secsec} we give some partial results generalizing this to higher rank vector bundles; see for instance Corollary \ref{corsecant}.

\subsection*{Acknowledgements} 
 The first author would like to thank Mihnea Popa for a number of conversations that led to work on this paper.
The authors would like to thank  S. Ramanan for a discussion concerning the stability of tensor products of vector bundles (his comments appear as Theorem \ref{teoram} and its proof). The authors would also like to thank the referee for many useful comments that have improved the content and exposition of this paper.

\section{Preliminaries}

\subsection{Singularities, tangent cones, and multiplicity}

When studying singularities of divisors, we will use the following basic fact:
\begin{lem}\label{lem11}
Let $D$ be an effective divisor on a smooth variety $X$.  Let $x\in X$ be a point lying on $D$, and suppose
 that $f:S\rightarrow X$ is the inclusion of a smooth curve, such that $f(s)=x$ for some $s\in S$.
 Then $\textnormal{mult}_s (D|_S) \ge \textnormal{mult}_x D$,
 and equality holds if and only if
$f_*T_s S \nsubseteq C_x D$.  It follows that, as a set, the tangent
cone
$$
C_x D=\{\alpha\in T_x X : \exists \textnormal{ a smooth curve } S\subseteq X, \textnormal{ with } $$ $$
T_x S = \alpha, \textnormal{ and }
 \textnormal{mult}_x(D|_S)> \textnormal{mult}_x D
\}
$$
\empty
    \end{lem}
\begin{proof}
This is straightforward to prove, but for a reference, and  a
generalization, see Fulton \cite{f} Corollary 12.4.
\end{proof}

For higher codimension subvarieties, there is the standard method of Kempf \cite{kt}, outlined in Arbarello
 et al. \cite{acgh}, which is paraphrased in the lemma below in the form we will be using it:

\begin{lem}[Kempf \cite{kt}, Arbarello et al.\cite{acgh}]\label{lemk}
Let $f:X\rightarrow Y$ be a proper morphism of smooth varieties  with image $Z$.  Let $z\in Z$, and denote by $X_z$
the scheme theoretic fiber $f^{-1}(z)$.  Let $N=N_{X_z/X}$ be the normal cone.
Suppose:
\begin{enumerate}
\item $f:X\rightarrow Z$ is birational.
\item $X_z$ is smooth.
\item There exist vector spaces $A$ and $B$, of dimensions $a$ and $b$ respectively, a linear map
$$\phi:A\otimes B\rightarrow (T_zY)^\vee,$$
and an integer $1\le w\le a$
 such that
\begin{enumerate}
\item $N=\{(W,\alpha)\in G(w,A)\times T_z Y:\alpha \perp \phi(W\otimes B)\}$,

\item With this description of $N$, $df:N\rightarrow T_zY$ is given by projection to the second factor.

\item For each $W\in G(w,A)$, the induced map
$$\phi_W:W\otimes B\rightarrow (T_zY)^\vee$$
 is injective.
\end{enumerate}
\end{enumerate}
Then:

\begin{enumerate}
\item $C_zZ$ is Cohen-Macaulay, reduced, and normal.
\item Taking bases $x_1,\ldots,x_a$ and $y_1,\ldots,y_b$ for $A$ and $B$ respectively,
 the ideal of $C_zZ$ is generated by the $(a-w+1)\times (a-w+1)$ minors of the matrix
$$
M=[\phi(x_i\otimes y_j)]_{\stackrel{i=1,\ldots,a}{j=1,\ldots,b}}.
$$
\item The multiplicity of $Z$ at $z$ is given by
$$
\textnormal{mult}_zZ = \prod_{h=0}^{w-1}\frac{(b+h)!h!}{(a-w+h)!(b-a+w+h)!}.
$$
\end{enumerate}

    \end{lem}
\begin{proof}
See Arbarello et al. \cite{acgh}, Chapter II Section 1 and Lemma p.242.
\end{proof}

\subsection{Cup products and extensions}
Given vector spaces $A$, $B$, and $C$, and a linear map
$\phi:A\otimes B \rightarrow C,
$
taking duals, and tensoring by $A$, gives $
A\otimes C^\vee \rightarrow A\otimes A^\vee \otimes B^\vee \stackrel{tr\otimes Id_{B^\vee}}{\rightarrow} B^\vee,
$
and thus we have an induced map
$
\psi:A\otimes C^\vee \rightarrow B^\vee.
$

Given vector bundles $E$ and $F$ on $C$, there are cup product maps
$$
H^i(E)\otimes H^j(F)\stackrel{\cup}{\rightarrow} H^{i+j}(E\otimes F)
$$
given locally by multiplication.

In particular, we have a map
$$
H^0(E)\otimes H^0(F^\vee \otimes K_C)\stackrel{\cup}{\rightarrow}
H^0(E\otimes F^\vee \otimes K_C).
$$
Taking duals, and using Serre duality, there is an induced map
$$
H^0(E)\otimes H^1(F\otimes E^\vee)\stackrel{\cup^\vee}{\rightarrow}
 H^1(F),
$$
which we will call the dual cup product.

Recall the standard result:

\begin{lem}\label{lemcupdet}
Let $E$ and $F$ be vector bundles on a smooth curve $C$, and suppose
$$
0\rightarrow F\rightarrow G \rightarrow E \rightarrow 0
$$
is an extension of vector bundles with class $\alpha\in H^1(E^\vee\otimes F)$.
 Then the connecting homomorphism in the long exact sequence in cohomology
$
\delta:H^0(E)\rightarrow H^1(F)
$
is given by the dual cup product with $\alpha$; i.e.
for $s\in H^0(E)$, $\delta(s)=s\cup^\vee \alpha$.

Moreover, let $w\le h^0(E)$ be a positive integer. After choosing bases for $H^0(E)$
and $H^0(F^\vee \otimes K_C)$, let $M$ be the matrix with entries in
$H^0(E\otimes F^\vee \otimes K_C)=H^1(E^\vee \otimes F)^\vee$ representing the cup product.
Then the set
$$
\{\alpha \in H^1(E^\vee \otimes F) : \exists \ W\in G(w,H^0(E)) \textnormal{ s.t. } W\cup^\vee \alpha=0\}
$$
is equal to the set
$$
\{\alpha \in H^1(E^\vee \otimes F) : \operatorname{rank}(M(\alpha))\le h^0(E)-w\}.
$$

    \end{lem}

\begin{proof}
See for example  Kempf \cite{ka}.
\end{proof}

\subsection{Families of vector bundles with sections} A family of vector bundles of rank $r$ and degree $d$
over a smooth curve $C$, parametrized by a scheme $S$, is a vector bundle $\mathscr{E}$ over $C\times S$
 such that, for all closed points $s\in S$, the restriction of $\mathscr{E}$ to $C\times \{s\}$,
 denoted $\mathscr{E}_s$, is a rank $r$, degree $d$ vector bundle over $C\times \{s\}$.
 For simplicity, we may sometimes use the notation ``family of vector bundles'' instead, when the rank,
  degree, and parameter space are clear.
Also, since $\chi(\mathscr{E}_s)$ is constant, we will often use $\chi$ to refer to this invariant.

\begin{lem}
Let $\mathscr{E}$ be a family of vector bundles over a smooth curve $C$, parameterized by a scheme $S$.
Then there is an exact sequence
$$
0 \rightarrow \pi_{2*}\mathscr{E} \rightarrow K^0 \stackrel{\gamma}{\rightarrow} K^1
\rightarrow R^1\pi_{2*}\mathscr{E} \rightarrow 0,
$$
where $\pi_2$ is the second projection $C\times S\rightarrow S$, and $K^0$ and $K^1$ are vector bundles such that
$\operatorname{rank}(K^1)=-\chi+\operatorname{rank}(K^0)$.
    \end{lem}

\begin{proof}
This is standard; see for example Arbarello et al. \cite{acgh}.
\end{proof}

Given $\mathscr{E}$ a family of vector bundles over a smooth curve $C$, parameterized by a scheme $S$,
 define $B^k(\mathscr{E})$ to be the $(\operatorname{rank}(K_0)-k)$-th degeneracy locus of the map $\gamma$
 given in the lemma.  Neither this map nor the complex are unique, but one can show that the degeneracy locus is
 independent of the choice of such a map and complex (it is the scheme defined by the $(k-\chi) $-th
 Fitting ideal of $R^1\pi_{2*}\mathscr{E}$, cf. Arbarello et al. \cite{acgh} p.179).  Clearly, as a set
$$
B^k(\mathscr{E})=\{s\in S: h^0(\mathscr{E}_s)\ge k\}.
$$
Given the description of $B^k(\mathscr{E})$ as a degeneracy locus, it
has an expected dimension we will call $\rho(g,k,\mathscr{E})$,
which is given by the formula
$$
\rho(g,k,\mathscr{E})=\dim(S)-k(-\chi+k).
$$

\begin{exa} Let $C$ be a smooth curve of genus $g$, let $d\in \mathbb{Z}$, let $S=\textnormal{Pic}^d(C)$,
and let $\mathscr{E} =\mathscr{L}_d$ be a Poincar\'e line bundle on $C\times \textnormal{Pic}^d(C)$, i.e.
 $\mathscr{L}_{[L]}\cong L$.  Then $B^k(\mathscr{L}_d)=B^k_d$, and
$$
\rho(g,k,\mathscr{L})=g-k(g-1-d+k),
$$
the familiar Brill-Noether number $\rho(g,k,d)$ (in the case of line
bundles, $r$ is usually used in place of $k-1$).
    \end{exa}

\subsection{Families of vector bundles over moduli spaces}\label{secfammod}
Given $\mathscr{E}$ a family of vector bundles over $C$
parameterized by a scheme $S$, the results of the previous section
allow us to give a ``good'' scheme structure on the set of points
$s\in S$ such that $h^0(\mathscr{E}_s)\ge k$, namely the scheme
$B^k(\mathscr{E})$.

Although there is no ``universal family'' of vector bundles over $C\times \mathscr{U}(r,d)$, i.e.
 a vector bundle $\mathscr{E}$ such that $\mathscr{E}_{C\times [E]}\cong E$, locally such a vector bundle exists,
  and this allows us to define a scheme structure on the set
$\{[E]\in \mathscr{U}(r,d):h^0(E)\ge k\}$, which we call
$B^k_{\mathscr{U}(r,d)}$. To be precise, for each $E\in \mathscr{U}$
there is an open affine neighborhood $U$ containing $E$, an \'etale
morphism $f:S\rightarrow U$, and a vector bundle $\mathscr{E}_S$
over $C\times S$ such that for each $s\in S$, $\mathscr{E}_s$ is
isomorphic to the vector bundle parameterized by $f(s)$.  This
allows us to define the scheme structure locally, and since the
construction is functorial, it defines a global scheme structure.

\subsection{Infinitesimal deformations of vector bundles}
Consider the artinian arc $S_a =\textnormal{Spec } \mathbb C [t]/(t^{a+1})$.
For $\ell < a$ the exact sequence
$$
0\rightarrow (t^{\ell+1}) \rightarrow \mathbb  C[t]/(t^{a+1})\rightarrow \mathbb C [t]/(t^{\ell+1})\rightarrow 0
$$
induces inclusions $S_\ell \rightarrow S_a$ (as closed subschemes), where $S_0$ corresponds to the unique closed point
 of $S_a$.

Given a vector bundle $E$ on a curve $C$, an $N$-th order deformation of $E$ is a vector bundle
$\mathscr{E}_N$ on $C\times S_N$ such that $\mathscr{E}_0\cong E$, where $\mathscr{E}_0$ is defined as the pullback of
$\mathscr{E}_N$ over the map
$C\times S_0\rightarrow C\times S_N$.
In other words, it is a family of vector bundles over $C\times S_N$ such that the fiber over the closed point of $S_N$
 is isomorphic to  $E$.

It is easy to check that the sheaves of sections fit into exact sequences
\begin{equation}\label{eqndef}
0 \rightarrow \mathscr{E}_a \stackrel{t^{b+1}}{\rightarrow} \mathscr{E}_{a+b+1} \rightarrow \mathscr{E}_b \rightarrow 0.
\end{equation}
In particular, given a first order deformation $\mathscr E $ of $E$, there is a corresponding extension:
\begin{equation}\label{eqnext1}
0\rightarrow E\rightarrow \mathscr{E} \rightarrow E \rightarrow 0.
\end{equation}
We say that a section $s\in H^0(E)$ lifts to first order (as a section of $\mathscr{E}$) if it is in the image of
the induced map
$$
H^0(\mathscr{E})\rightarrow H^0(E).
$$
We say that a subspace $W\subseteq H^0(E)$ lifts to first order if it is contained in the image of this map.
 More generally, we say that a section, or subspace of $ H^0(\mathscr{E}_b)$ lifts to order $a+b+1$
 if it is in the image of the map induced from the exact sequence \eqref{eqndef} above.

\begin{lem}\label{lem16}
Let $\mathscr{E}$ be a first order deformation of a vector bundle $E$ on a smooth curve $C$, corresponding to
$\alpha \in H^1(E^\vee \otimes E)$.
Then a section $s\in H^0(E)$ lifts to first order as a section of $\mathscr{E}$ if and only if
$$
s\cup^\vee \alpha = 0\in H^1(E).
$$
    \end{lem}
\begin{proof}
This follows from the exact sequence (\ref{eqnext1}) above, since we have seen that the coboundary map is given by the dual cup product with the extension class.
\end{proof}

\subsection{Coherent systems} In this section we review certain facts about coherent systems which we will need in what follows.  We refer the reader to Bradlow et al. \cite{n}, King-Newstead \cite{kn}, He \cite{he}, Le Potier \cite{lepot} and Raghavendra-Vishwanath \cite{ravi} for more details.  Roughly, a coherent system of type $(k,d,r)$ is a pair $(E,V)$ with $E$ a vector bundle of rank $r$ and degree $d$ and $V\in G(k,H^0(E))$.  To be more precise, one should consider triples  $(E,\mathbb{V},\phi)$ such that $E$ is a vector bundle of rank $r$ and degree $d$,  $\mathbb{V}$ is an $k$-dimensional vector space, and $\phi:\mathbb{V}\otimes \mathscr{O}_C\rightarrow E$ is a morphism of coherent sheaves such that the induced map on global sections is an inclusion.  $V$ would then be the image of $H^0(\phi)$.  We will write coherent systems as pairs $(E,V)$ for simplicity.

Given $\alpha\in \mathbb{R}$, we define the $\alpha$-slope of a coherent system as $$ \mu_\alpha(E,V)=\frac{d}{r}+\alpha \frac{k}{r}. $$ We say that a coherent system $(E,V)$ is $\alpha$-stable  if $\mu_\alpha(E',V')<\mu_\alpha(E,V)$ for all proper subcoherent systems $(E',V')\subset (E,V)$.  We say that the coherent system is semi-stable if the same condition holds with $<$ replaced by $\le$.  One can define the notion of a family of $\alpha$-(semi)stable coherent systems, parametrized by a scheme $S$, and a notion of morphism of $\alpha$-(semi)stable coherent systems in such a way that for all $\alpha$, there exist (possibly empty)  coarse moduli spaces of  $\alpha$-stable coherent systems (King-Newstead \cite{kn}).  We will call these spaces $\bar{G}^k_{\alpha,\mathscr{U}}$.

When $\alpha=0$ and $r>0$, it
is clear that there are no strictly $\alpha$-stable coherent
systems. On the other hand, there exists an $\alpha_1>0$ such that
for all $0<\alpha<\alpha_1$, $E$ stable implies that  $(E,V)$ is
$\alpha$-stable for all $V\in G(k,H^0(E))$ (e.g. \cite[Proposition
2.5]{n}, \cite{kn}, \cite{dask}).   The moduli spaces $\bar{G}^k_{\alpha,\mathscr U}$ are
isomorphic for all $0<\alpha<\alpha_1$, and we will call this space
$\bar{G}^k_{\mathscr U}$ (cf. \cite[pp.689-90]{n}).

On the other hand, for $0<\alpha<\alpha_1$, it is known that $(E,V)$ $\alpha$-stable implies that $E$ is semi-stable (e.g. \cite[Proposition 2.5]{n}, \cite{kn},\cite{dask}).  Consequently there is a
natural map
$$a:\bar{G}^k_{\mathscr{U}}\rightarrow\overline{\mathscr{U}}$$
given by forgetting the vector space,  where
$\overline{\mathscr{U}}$ is the moduli space of semi-stable vector
bundles of rank $r$ and degree $d$.   We will also denote by $a:G^k_{\mathscr U}\to \mathscr U(r,d)$ the restriction of this forgetful map, where
as a set $$G^k_{\mathscr U}
=\{(E,V):E\in B^k_{r,d}, \  V\in G(k,H^0(E))\}.$$

Given two coherent
systems $(E,V)$ and $(E',V')$, there is an exact sequence
(\cite[Corollaire 1.6]{he}, cf. \cite[Section 3]{n}) $$ 0\rightarrow
Hom((E',V'),(E,V)) \rightarrow Hom(E',E)$$ $$ \rightarrow
Hom(V',H^0(E)/V) \rightarrow Ext^1((E',V'),(E,V)) \rightarrow
Ext^1(E',E)  $$ $$\rightarrow Hom(V',H^1(E)) \rightarrow
Ext^2((E',V'),(E,V)) \rightarrow 0.$$  Thus if $(E',V')=(E,V)$ and $E$
is stable, there is an exact sequence $$ 0\rightarrow
Hom(V,H^0(E)/V) \rightarrow Ext^1((E,V),(E,V)) \rightarrow
T_E\mathscr{U}.$$ 
It is a fact \cite[Th\'eor\`eme 3.12]{he}  that
if $(E,V )$ is $\alpha$-stable, then
$$T_{(E,V)}G^k_{\alpha,\mathscr{U}}=Ext^1((E,V),(E,V))$$ 
and in
this case we get   $$ 0\rightarrow Hom(V,H^0(E)/V) \rightarrow
T_{(E,V)}G^k_{\alpha,\mathscr{U}} \stackrel{a_*}{\to} T_E\mathscr{U}.
$$  This can be interpreted in the following way, which is now
standard and therefore we omit the proof; see for example
\cite[Th\'eor\`eme 3.12]{he} and \cite[\S 3 and Proposition
3.10]{n}.

 \begin{pro}\label{protan} Let $E\in B^k_{r,d}$. \begin{enumerate} \item $\operatorname{image}(a_*)=(\operatorname{image} (\mu_V))^\perp$, where $$ \mu_V:V\otimes H^0(E^\vee \otimes K_C)\rightarrow  H^0(E\otimes E^\vee \otimes K_C)  $$ is the restriction of the cup product to $V$. \item $G^k_{\mathscr{U}}$ is smooth and of dimension $\rho(g,k,\mathscr{U})$ at $(E,V)$ if and only if  $\mu_V$ is injective.  \end{enumerate} \end{pro}

\subsection{Preliminaries on theta divisors}
Given $\mathscr{E}$ a family of vector bundles of rank $r$ and
degree $r(g-1)$ over a smooth curve $C$, parametrized by a scheme
$S$, we define $\Theta_S := B^1(\mathscr{E})$;  in other words $\Theta_S$ is the subscheme of
$S$ parametrizing vector bundles with sections.  As pointed out
earlier, there is a complex
$$
0\rightarrow K^0\stackrel{\gamma}{\rightarrow} K^1 \rightarrow 0
$$
of locally free sheaves on $S$ of equal rank, such that $\Theta_S$ is defined as the zero locus of the determinant
 of $\gamma$.  Thus if $\Theta_S\ne S$, then it is a divisor.

It will be important for us to study the case that $S$ is a smooth
curve or an artinian arc $S_l=\textnormal{Spec } \mathbb
C[t]/(t^{l+1})$.    The following is straightforward:
\begin{lem}\label{lem18}
If $S$ is a smooth curve, and $\Theta_S$ is a divisor on $S$, then
$\Theta_S=\sum_{s\in S}\ell((R^1\pi_{2*}\mathscr{E})_{s})\cdot s$.
    \end{lem}
\begin{proof}
See \cite{cmf} Section 1.1.
\end{proof}

We would like a way to calculate
$\ell((R^1\pi_{2*}\mathscr{E})_{s})$.  We will use the notation
$C_l$ to denote $C\times S_l$, and $\mathscr{E}_l$ to denote the
pullback of $\mathscr{E}$ to $C_l$.

\begin{lem}\label{lem19}
If $S$ is a smooth curve, and $\Theta_S$ is a divisor on $S$, then
there exists an integer $N\ge 0$ such that for all $l\ge N$
$$\ell((R^1\pi_{2*}\mathscr{E})_{s})=
\ell (H^0(C_k,\mathscr{E}_l))=h^0(C,\mathscr{E}_l).$$ Moreover, if
for any $l\ge 0$, $H^0(\mathscr{E}_l)=H^0(\mathscr{E}_{l+1})$, then
we may take $N=l$.
    \end{lem}
\begin{proof}
See \cite{cmf} Section 1.2.
\end{proof}

We then have the following lemma:

\begin{lem}\label{lemtheta}
Let $\mathscr{E}$ be a family of vector bundles with $\chi=0$ over a
smooth curve $C$, parametrized by a smooth variety $X$, and such
that $\Theta_X$ is a divisor. Then
 $$\textnormal{mult}_x \Theta_X\ge h^0(C,\mathscr{E}_x),$$
  and equality holds if and only if there is a tangent vector $\alpha\in T_x X$ such that
$$H^0(\mathscr{E}_\alpha)=H^0(\mathscr{E}_x);$$
i.e. $\textnormal{mult}_x \Theta_X = h^0(C,\mathscr{E}_x)$ if and only if there exists a tangent vector
 $\alpha$ such that no non-zero sections of $H^0(\mathscr{E}_x)$  lift to first order as sections of
$H^0(\mathscr{E}_\alpha)$.
    \end{lem}
\begin{proof}
Restricting to a general smooth curve $S\subseteq X$ passing through $x$, the previous lemmas imply
 $\textnormal{mult}_x \Theta_X=\textnormal{mult}_x \Theta_X|_S =h^0(C,\mathscr E_N)\ge  h^0(C,\mathscr{E}_x)$,
 for some $N>>0$.  In addition, Lemma \ref{lem19} implies that if $h^0(C,\mathscr E_1)=h^0(C,\mathscr E_x)$,
 then $h^0(C,\mathscr E_N)=h^0(C,\mathscr E_x)$.
\end{proof}

We can rephrase this using cup products.
Let $\mathscr{E}$ be a family of vector bundles with $\chi=0$ over a smooth curve $C$, parameterized
 by a smooth variety $X$, and such that
$\Theta_X$ is a divisor.  This family induces a map
$f:X\rightarrow \mathscr{U}$, which for each $x\in X$ induces a map
 $f_*:T_x X \rightarrow H^1(\mathscr{E}_x^\vee \otimes
\mathscr{E}_x)$.

\begin{cor}\label{corcup}
In the notation above, $\textnormal{mult}_x \Theta_X =
h^0(C,\mathscr{E}_x)$  if and only if there exists a tangent vector
$\alpha\in T_x X$ such that the dual cup product map
$$
\cup^\vee f_*\alpha: H^0(\mathscr{E}_x)\rightarrow
H^1(\mathscr{E}_x)
$$
is injective.  Moreover, if
$\textnormal{mult}_x \Theta_X =
h^0(C,\mathscr{E}_x)$  then
$$C_x\Theta =\{\alpha\in T_xX:(\ker \cup^\vee f_*\alpha) \ne 0\}.
$$
    \end{cor}
\begin{proof}
This is just a translation of the previous lemmas.
\end{proof}

We will also use the following lemma:

\begin{lem}\label{lemcomp} Let $E$ be a vector bundle on $C$.
Let $D$ be an effective divisor of degree $d$, and let
$\delta:H^0(\mathscr O_D(D))\rightarrow  H^1(\mathscr O_C)$ be the
induced coboundary map. Suppose $\epsilon\in H^0(\mathscr O_D(D))$,
and $\delta(\epsilon)=\alpha\in H^1(\mathscr O_C)$.   Then the dual
cup product $\cup ^\vee \alpha:H^0(E)\rightarrow H^1(E)$ is given by
composing $\cup \epsilon: H^0(E)\rightarrow H^0(E(D)|_D)$, with the
coboundary map $\partial:H^0(E(D)|_D)\rightarrow H^1(E)$.
\end{lem}
\begin{proof}
See \cite{cmf} Section 1.2.
\end{proof}

\subsection{A remark on multiplier ideals}\label{multid}  In this section we apply a result of Ein-Lazarsfeld
\cite{el} for pluri-theta divisors on abelian varieties to the case of certain generalized theta divisors on
moduli of vector bundles over curves.  For definitions, and more details on multiplier ideals,
we refer the reader to Lazarsfeld \cite{laz}.

Recall that given $\mathscr U_r=\mathscr U(r,r(g-1))$, the subvariety
$B^1_{r,r(g-1)}$ is a divisor which we will denote $\Theta_r$.
 Given a general
$E\in \mathscr U_r$, it is known  that the map
$$i_E:\textnormal{Pic}^0(C)\rightarrow \mathscr U_r$$
 defined  by $\xi
\mapsto \xi \otimes E$ is an embedding (see \cite{Li}).
It is easy to check by degenerating $E$ to a direct sum of $r$ line
bundles that if  $\Theta_r|_{\textnormal{Pic}^0(C)}=\Theta_E$ is a
divisor,
 then $\Theta_E\in |r\Theta|$, where $\Theta$ is a translate of the Riemann theta divisor on $\textnormal{Pic}^0(C)$.

Now, given a $\mathbb Q$-divisor D on a smooth quasi-projective
variety $X$, and $c\in \mathbb Q^+$, let
$\mathcal{J}(X,cD)=\mathcal{J}(cD)\subseteq \mathscr O_X$ be the associated multiplier
ideal.
 We will say that the pair $(X,D)$ is log-canonical if $\mathcal{J}(X,(1-\epsilon)D)=\mathscr O_X$
  for all $0< \epsilon <1$.  The pair is said to be log-canonical at a point $x\in X$ if there is an open neighborhood $U$ of $x$ so that the pair $(U,D|_U)$ is log-canonical.

\begin{rem}\label{remel}
A direct consequence of the results in \cite{el} is that
for all $r\ge 1$,
the pair $(\mathscr U_r,\frac{1}{r}\Theta_r)$ is log-canonical at all points $E$ such that
$i_E:\textnormal{Pic}^0(C)\rightarrow \mathscr U_r$ is an embedding and $\Theta_E\ne \operatorname{Pic}^0(C)$.
\end{rem}

Indeed, for $(A,\Theta)$ a principally polarized abelian variety (ppav), and $D\in |r\Theta|$,
 it is a result of Ein and Lazarsfeld \cite{el} that $(A,\frac{1}{r}D)$ is log-canonical.
 Setting $  \mathscr U_r'=\{E\in \mathscr U_r :   i_E \textnormal{ is an embedding and } \Theta_E\ne \operatorname{Pic}^0(C)\}$, the remark then follows from the fact (cf. Lazarsfeld \cite[Corollary 9.5.6]{laz} ) that
for all $E\in \mathscr U_r'$,
$$\mathcal{J}(\textnormal{Pic}^0(C),(1-\epsilon)\frac{1}{r}\Theta_E)\subseteq
\mathcal{J}(\mathscr U_r',(1-\epsilon)\frac{1}{r}\Theta_r')|_{\textnormal{Pic}^0(C)}.$$

\begin{rem}
It follows from Remark \ref{remel} (see for example \cite[Example 9.3.10]{laz}) that $\operatorname{mult}_E\Theta_r\le r[r^2(g-1)+1]$
 for all points $E$ such that
$i_E:\textnormal{Pic}^0(C)\rightarrow \mathscr U_r$ is an embedding and $\Theta_E\ne \operatorname{Pic}^0(C)$.  This bound is far from the bounds obtained using Theorem \ref{teo1}, and  ``Clifford'' type theorems.
For instance, using a  special case
of Theorem \ref{teo1}, Laszlo  showed that for $E\in \mathscr U(2,2(g-1))$, if $C$ is not hyperelliptic,
 then $\operatorname{mult}_E \Theta_2\le g$  (\cite[Remark after Proposition IV.2]{l} ).  This is similar to the difference between the bounds  on singularities of theta divisors on abelian varieties given via multiplier ideals, and those for Jacobians  given via Clifford's Theorem.
\end{rem}

\section{Singularities of Brill-Noether loci}\label{seckempf}

\subsection{Singularities of Brill-Noether loci for vector bundles parameterized by $\mathscr{U}(r,d)$.}

In this section we study the singularities of $B^k_{r,d}$.  We begin
by studying the tangent spaces to $G^k_{r,d}$.

In what follows, for a vector bundle
$E\in \mathscr{U}$, and
a subspace $W\subseteq H^0(E)$, we denote by $\mu_W$ the restriction of the cup product map to $W$:
$$\mu_W:W \otimes H^0(E^\vee\otimes K_C)\rightarrow H^0(E^\vee\otimes E \otimes K_C).$$

\begin{pro} Suppose that $0\le \rho(g,k,\mathscr{U})< \dim(\mathscr{U})$.
Let $E\in B^k_{r,d}$, and suppose that for all $W\in G(k,H^0(E))$,
 $\mu_W$ is injective.  Then $G^k_{r,d}$ is smooth in a neighborhood of the fiber $a^{-1}(E)$, and
 $a^{-1}(E)$ is smooth as a scheme.  Moreover, the normal  bundle $N=N_{a^{-1}(E)/G^k_{r,d}}$ can be described as
$$
 N=\{(W,\alpha)\in G(k,H^0(E))\times H^1(E^\vee \otimes E): \alpha\perp \operatorname{image}(\mu_W) \}
$$
and the differential $a_*:N\rightarrow T_E\mathscr{U}=H^1(E^\vee \otimes E)$ is given by projection onto the second factor.
    \end{pro}
\begin{proof}
The description of the normal bundle and the differential are clear from Proposition \ref{protan}.
  The fact that the scheme theoretic fiber is smooth follows from the injectivity of the differential on the normal bundle.
\end{proof}

\begin{cor} With the same assumptions as in the proposition, if the cup product
$\mu_W$ is injective for all $W\in G(k,H^0(E))$, then $E$ lies on a
unique irreducible component $Z\subseteq B^k_{r,d}$, and
$Z\nsubseteq B^{k+1}_{r,d}$. Moreover, the morphism
$a:G^k_{r,d}\rightarrow B^k_{r,d}$  is birational in a neighborhood
of $E$.    \end{cor}
\begin{proof}The fibers of $a$ are Grassmannians, and hence connected.
If there were more than one component of $B^k_{r,d}$ passing through
$E$, there would have to be multiple components of the fiber over
$E$, since $G^k_{r,d}$ is smooth along the fiber.  This would
contradict the connectedness of the fiber. Now let $Z$ be the unique
irreducible component containing $E$.  It must have dimension at
least $\rho(g,k,\mathscr{U})$.  If $Z\subseteq B^{k+1}_{r,d}$, then
$\dim(a^{-1}(Z))>\rho(g,k,\mathscr{U})$, since the Grassmannian has
positive dimension.  But this contradicts the fact that
$\dim(G^k_{r,d})=\rho$ along $a^{-1}(E)$ (Proposition
\ref{protan}).
\end{proof}

\begin{rem}  In fact no component of $B^k_{r,d}$ is contained in $
B^{k+1}_{r,d}$ (see  \cite[Proposition 1.6]{laum}).
\end{rem}

To state the following theorem, we need to introduce some notation.
First we must choose  bases $x_1,\ldots,x_{h^0(E)}$ and
$y_1,\ldots,y_{h^1(E)}$ for $H^0(E)$ and $H^0(E^\vee\otimes K_C)$
respectively.   For simplicity, write $x_iy_j$ for the image of
$x_i\otimes y_j$ under the cup product mapping. We then have the
following generalization of a theorem of Laszlo:

\begin{teo}[Kempf's Theorem]\label{kempfteo} Suppose that $0\le \rho(g,k,\mathscr{U})< \dim(\mathscr{U})$.
Let $E\in W^k_{r,d}$, and suppose that for all $W\in G(k,H^0(E))$,
$\mu_W$ is injective.

 Then:
 \begin{enumerate}
\item $C_EB^k_{r,d}$ is Cohen-Macaulay, reduced and normal.

\item The ideal of $C_EB^k_{r,d}$, as a subvariety of $H^1(E^\vee\otimes E)$,
 is generated by the $(h^0(E)-k+1)\times (h^0(E)-k+1)$ minors of the matrix
$$
M=[x_iy_j]_{i=1,\ldots,h^0(E), j=1,\ldots,h^1(E)}.
$$
\item The multiplicity of $B^k_{r,d}$ at $E$ is
$$
\textnormal{mult}_E B^k_{r,d}=\prod_{h=0}^{k-1}
\frac{(h^1(E)+h)!h!}{(h^0(E)-k+h)!(h^1(E)-h^0(E)+k+h)!}.
$$

\item As a set $$C_E B^k_{r,d} = \bigcup_{W\in G(k,H^0(E))}(\operatorname{image}(\mu_W))^\perp.$$
\end{enumerate}
\end{teo}

\begin{proof}
With the previous results in this section, the theorem follows from
the results of Kempf \cite{kt}  (referenced in Lemma \ref{lemk}).
\end{proof}

\begin{rem}
For another interpretation of the tangent cones, see
Hitching \cite[Section 5]{h}.
\end{rem}

In the case that $k=1$, the restricted cup product maps are always
injective, allowing us to prove  the corollary below, of which
Theorem \ref{teo1} is a special case.

\begin{cor}
  Suppose $0\le \rho(g,1,\mathscr{U}(r,d))< \dim(\mathscr{U}(r,d))$.   Then $B^1_{r,d}$ is a reduced,
   Cohen-Macaulay subscheme of the expected dimension $\rho(g,1,\mathscr{U}(r,d))$; i.e.
$$ \dim B^1_{r,d}= \dim\mathscr{U}(r,d)-1+\chi. $$
If $\dim B^2_{r,d}\le \rho(g,2,\mathscr U(r,d))+1-\chi$, then
$B^1_{r,d}$ is also normal.

For all  $E\in B^1_{r,d}$, the tangent cone $C_E B^1_{r,d}$ is a
reduced, Cohen-Macaulay and normal subvariety of $T_E\mathscr
U(r,d)=H^1(E^\vee\otimes E)$, defined by the $h^0(E)\times h^0(E)$
minors of an $h^0(E)\times h^1(E)$ matrix with entries in
$H^1(E^\vee \otimes E)^\vee$. Moreover, $ \textnormal{mult}_E
B^1_{r,d}=\binom{h^1(E)}{h^0(E)-1}$.
    \end{cor}

\begin{proof}
The statement about the multiplicity follows directly from  Theorem
\ref{kempfteo}. It also follows from the injectivity of the cup
product  that $B^1_{r,d}$ is locally irreducible.
 Since $G^1_{r,d}$ is smooth of the correct dimension, $B^1_{r,d}$ also has the correct dimension;
  by virtue of the fact that $B^1_{r,d}$ is  determinantal, it follows that it is Cohen-Macaulay.

We have also seen that no component of $B^1_{r,d}$ is contained in
$B^2_{r,d}$, and so, by the formula for multiplicity, this shows
that $B^1_{r,d}$ is generically reduced. Generically reduced and
Cohen-Macaulay imply reduced.  If $B^1_{r,d}$ is non-singular in
codimension one, then it is normal.  It is easy to check that this
is the case if $\dim(B^2_{r,d})\le \rho(g,2,\mathscr
U(r,d))+1-\chi$.
\end{proof}

\begin{rem} It follows from work of Bradlow et al. \cite{n} that
when $B^1_{r,d}$ is normal, it has rational singularities, and
moreover, $B^1_{r,d}$ is normal for general curves. Indeed,
\cite[Theorem 11.7]{n}  implies that $G^1_{r,d}$ is a resolution of
$B^1_{r,d}$.  The fibers of the map $a$ are projective spaces, so
that $R^i a_* \mathscr O_{G^1_{r,d}}=0$ for all $i>0$. Thus when
$B^1_{r,d}$ is normal, $\mathbb Ra_*\mathscr O_{G^1_{r,d}}=\mathscr
O_{B^1_{r,d}}$, and so $B^1_{r,d}$ has rational singularities.
\cite[Theorem 8.1]{n} establishes that this is true for (Petri)
general curves.
\end{rem}

\begin{rem}\label{remgencup}
 Recall that given a vector bundle $E\in \mathscr U(r',d')$, one can also consider the subvarieties
 $B^k_E \subseteq \mathscr U(r,d)$ parameterizing those vector bundles that when tensored by $E$
  have at least $k$ linearly independent global sections; i.e. as a set $
B^k_E=\{M\in \mathscr U(r,d):h^0(E\otimes M)\ge k\}. $
 Given $W\subseteq G(k, H^0(E\otimes M))$, let
$$
\mu'_W:W\otimes H^0(E^\vee \otimes M^\vee \otimes K_C) \rightarrow H^0(M^\vee\otimes M\otimes K_C)
$$
be the appropriate cup product map. If $E\otimes M$ is stable, which is the case for general $E$ (see Theorem \ref{teoram} below), then the same statement as in
Theorem \ref{kempfteo}  holds for $B^k_E\subseteq \mathscr U(r,d)$ and
$C_MB^k_E$ in this situation, replacing $E$ with $E\otimes M$, and
$\mu_W$ with $\mu'_W$.  
\end{rem}

\begin{rem}\label{remcup}
Unlike in the case of $B^k_{r,d}$, for the varieties $B^k_E\subseteq
\textnormal{Pic}^0(C)$, even when $\chi(E)=0, k=1$, it is not clear
when the dual cup product map is injective.  In fact, if it were
always injective, it would follow that for any $L\in
\textnormal{Pic}^0(C)$ and any $E\in \mathscr U(r,r(g-1))$,
$\textnormal{mult}_L B^1_E=h^0(E\otimes L)$. Examples where this
condition fails were constructed by Laszlo \cite{l} and the second
author \cite{t}.  See Remark \ref{remex} and Section \ref{secexc}
for more on this.
\end{rem}

The proof of the following result was communicated to the authors by
S. Ramanan.

\begin{teo} \label{teoram} Let $M$ be
a stable vector bundle on $C$ and $E$ a  generic stable vector
bundle on $C$. Then  $M\otimes E$ is stable.
\end{teo}

\begin{proof} By virtue of the tensor product preserving correspondence between stable vector
bundles and representations (see \cite{ns}, \cite{ramanan}), it follows that  
$M\otimes E$ is polystable. Consequently, it is well known that in order to
check the stability of $M\otimes E$, it suffices to show that its only automorphisms are homotheties, that is $h^0(M\otimes E\otimes
M^\vee\otimes E^\vee)=1$.

To begin, we recall that the trace map $Tr:E\otimes E^\vee\to \mathcal O_C$ gives rise to the well known decomposition
\begin{equation}\label{eqndecomp}
E\otimes E^\vee={\mathcal O}_C\oplus ad(E)
\end{equation}
where ${\mathcal O}_C$
represents the homotheties of $E$ and $ad(E)$ the set of traceless
automorphisms.  Concretely, the splitting is given by sending an automorphism
$\varphi$ of $E$ to the pair $(Tr(\varphi), \varphi-\frac{Tr(\varphi)}{\operatorname{rk E}}Id_E)$. By assumption $E$ is stable, and thus $h^0(E\otimes E^\vee)=1$. 
Hence,
$h^0(ad(E))=0$ (and the same applies to $ad(M)$).
We also point out that as $E$ and $E^\vee$ are both stable, one has (as above) that $E\otimes
E^\vee$ is polystable, and it then follows from (\ref{eqndecomp}) that $ad(E)$ is  polystable as well (and the same applies to $(ad(M))^\vee$).

Now the decomposition (\ref{eqndecomp}) gives rise to a decomposition
$$M\otimes E\otimes  M^\vee\otimes E^\vee={\mathcal O}_C\oplus
ad(M)\oplus ad(E)\oplus ad(M)\otimes ad(E).$$
We already know that
$h^0(ad(M))=h^0(ad(E))=0$. So in order to check that  $h^0(M\otimes E\otimes
M^\vee\otimes E^\vee)=1$, it only remains to show that
$h^0(ad(M)\otimes ad(E))=0$. This is equivalent to showing there
are no non-trivial morphisms $(ad(M))^\vee \to ad(E)$. 
As the two
bundles are polystable of degree zero, this is equivalent to proving
that none of their direct summand decompositions are isomorphic. This
is true if $E$ is generic, since one can for instance
degenerate $E$ to a direct sum.
\end{proof}

\section{Singularities of generalized theta divisors}
In this section we consider the multiplicity of singularities of generalized theta divisors.
  Since Laszlo's result (cf.
Theorem \ref{teo1}) gives a complete description of the
singularities for $\Theta_r \subseteq \mathscr U(r,r(g-1))$,  we
will focus on the restrictions of these divisors.

\subsection{Singularities of $\Theta_E$ in $\mathscr  U(r,d) $}
In this section we generalize the results of Laszlo \cite{l} for $|2\Theta|$ divisors on
$Pic^0(C)$ to all generalized theta divisors on
$\mathscr U(r,d)$.
Recall that given $E\in \mathscr U(r',d')$, such that
$r'd+rd'=rr'(g-1)$, we define
$$\Theta_E:=\{M\in \mathscr U(r,d) : h^0(E\otimes M)>0\}.
$$
In what follows we will always make the assumption that $\Theta_E$ is a divisor.
We have seen:
\begin{pro}  In the notation above
$$
\textnormal{mult}_M\Theta_E\ge h^0(E\otimes M),
$$
and $\textnormal{mult}_M \Theta_E=h^0(E\otimes M)$ if and only if
there exists a cocycle  $\alpha\in H^1(M^\vee \otimes M)$ such that
$\cup^\vee \alpha:H^0(E\otimes M)\rightarrow H^1(E\otimes  M)$ is
injective. Moreover, if $\textnormal{mult}_M\Theta_E=h^0(E\otimes
M)$, then the tangent cone $C_M\Theta_E$ is given as the determinant
of an $h^0(E\otimes M)\times h^0(E\otimes M)$ matrix with linear
entries.
     \end{pro}
\begin{proof}
This follows directly from the discussion in Section 1, and from Theorem \ref{teo1}.
\end{proof}

 We begin by looking at conditions for when
 $\textnormal{mult}_M\Theta_E=h^0(E\otimes M)$.

\begin{lem}\label{lemfirst}
Let $E$, $M$, and $\Theta_M$ be as above.  For any extension of vector bundles
$$
0\rightarrow L \rightarrow E \rightarrow F \rightarrow 0,
$$
$\textnormal{mult}_M\Theta_E\ge h^0(L\otimes M)+h^1(F\otimes M)$.
     \end{lem}

\begin{proof}  Consider an extension of vector bundles
$$
0\rightarrow L \rightarrow E \rightarrow F \rightarrow 0.
$$
Denote by $\mathscr{M}$ a first order deformation
$$
0\rightarrow M \rightarrow \mathscr{M} \rightarrow M\rightarrow 0,
$$
corresponding to $\alpha\in H^1(M^\vee\otimes M)$.  We have a commutative diagram
\begin{equation}\label{eqnbig}
\begin{CD}
 @. 0 @. 0 @.0 @.\\
 @. @VVV @VVV @VVV\\
0 @>>> L \otimes M @>>> E\otimes M @>>>F\otimes M @>>>0\\
@. @VVV @VVV @VVV\\
0 @>>> L\otimes \mathscr{M} @>>> E\otimes \mathscr{M} @>>>F\otimes \mathscr{M} @>>>0\\
@. @VVV @VVV @VVV\\
0 @>>> L\otimes M @>>> E \otimes M @>>>F \otimes M@>>>0\\
@. @VVV @VVV @VVV\\
 @. 0 @. 0 @.0 @.\\
\end{CD}
\end{equation}
where all the rows and columns are exact.  Let $e\in H^1(F^\vee
\otimes L)$ be the extension class of $E$.   The diagram above gives
rise to a diagram of long exact sequences, of which the following is
a part

\begin{equation}\label{eqnsma}
\begin{CD}
0@>>> H^0(L\otimes \mathscr{M}) @>>> H^0(E \otimes \mathscr{M}) @>>>H^0(F \otimes \mathscr{M}) \\
@. @VVV @VVV @VVV\\
0 @>>> H^0 (L\otimes M) @>>> H^0 (E\otimes M) @>>>H^0(F\otimes M) \\
@. @V\cup^\vee \alpha VV @V\cup^\vee \alpha VV @V\cup^\vee \alpha VV\\
 @>\cup^\vee e >> H^1(L\otimes M) @>>> H^1 (E\otimes M) @>>>H^1(F\otimes M).\\
\end{CD}
\end{equation}

Set
$$
K_{L\otimes M,\alpha}=\ker\left(\cup^\vee \alpha:H^0(L\otimes M)\rightarrow H^1(L\otimes M)\right).
$$

We must have
$$
\dim \left(\cup^\vee \alpha H^0(L\otimes M) \cap \cup^\vee  e H^0(F\otimes M)\right)
\ge
$$
$$
\left(h^0(L\otimes M)-\dim K_{L\otimes M,\alpha}\right)+
$$
$$\left(h^0(F\otimes M)-h^0(E\otimes M)+h^0(L\otimes M)\right) -h^1(L\otimes M).
$$
Since $\chi (L\otimes M)+\chi (F\otimes M)=\chi (E\otimes M)=0$, we
get

$$
\dim \left(\cup^\vee \alpha H^0(L\otimes M) \cap \cup^\vee e H^0(F\otimes M)\right)
\ge
$$
$$
h^0(L\otimes M)-(h^0(E\otimes M)-h^1(F\otimes M))-\dim K_{L\otimes M,\alpha}.
$$

On the other hand, focusing on the subspace $H^0(L\otimes
M)\subseteq H^0(E\otimes M)$,  and chasing through the bottom left
hand corner of the diagram above, we have
$$h^0(E\otimes \mathscr M)\ge
$$
$$\dim K_{L\otimes M,\alpha}+\dim \left(\cup^\vee \alpha H^0(L\otimes M)
\cap \cup^\vee e H^0(F\otimes M)\right)+ h^0(E\otimes M),$$
where the last term comes from the inclusion $H^0(E\otimes M)\to H^0(E\otimes \mathscr M)$.
We conclude that
$h^0(E\otimes \mathscr M)\ge
h^0(L\otimes M)+h^1(F\otimes M)$.
\end{proof}

\begin{rem} \label{remex}
In light of these results, a  natural question to ask is the following.  Suppose
$\textnormal{mult}_M \Theta_E>h^0(E\otimes M)$.   Is it true that
there must exist an extension
$$
0\rightarrow L \rightarrow E \rightarrow F \rightarrow 0
$$
such that  $h^0(L\otimes M)+h^1(F\otimes M)>h^0(E\otimes M)$? It
turns out this is not the case. To show this, we turn to an example
of Laszlo\footnote{The first author learned of this example from
Mihnea Popa  who pointed out that the vector bundle $E$ can be
viewed as a so called Lazarsfeld bundle, i.e. the dual of the kernel
of the evaluation map $H^0(K_C)\otimes \mathscr O_C
\stackrel{ev}{\to} K_C$.} \cite[Proposition IV.9]{l}: let $C$ be a
smooth non-hyperelliptic curve of genus $3$, and let $E$ be the
unique rank two stable vector bundle on $C$ such that $h^0(E)=3$,
and $\det(E)=K_C$. Given any extension of vector bundles
$$
0\rightarrow L \rightarrow E \rightarrow K_C\otimes L^\vee \rightarrow 0,
$$
it follows that $\deg(L)<g-1=2$, and thus $h^0(L)\le 1$. We also have $h^1(K_C\otimes L^\vee)=h^0(L)\le 1$, so that
$h^0(L)+h^1(K_C\otimes L^\vee)<h^0(E)$.  Nevertheless, it is easy to see that $\Theta_E=C-C$, and thus
$$\textnormal{mult}_{\mathscr O_C}\Theta_E=4>h^0(E).
$$
In fact it is an observation of Laszlo \cite[Remark V.8]{l} that
whenever $\det (E)=K_C$ and $h^0(E)$ is odd,  then
$\operatorname{mult}_{\mathscr O_C}\Theta_E>h^0(E)$.  It would be
nice to know of examples where this strict inequality holds that do
not fall into one of the two cases above.
     \end{rem}

We next show that when $h^0(E\otimes M)\le 2$, this problem does not
arise.  We start  with the following discussion. Given $s\in
H^0(E\otimes M)$, there is a corresponding morphism $M^\vee
\rightarrow E$.  Let $\phi:M^\vee \otimes M\rightarrow E\otimes M$
be the morphism obtained by tensoring with $M$.  It is easy to check
that $1_M\in H^0(M^\vee \otimes M)$ maps to $s\in H^0(E\otimes M)$.

\begin{lem}
Let $E$ and $M$ be as above.  Let $s\in H^0(E\otimes M)$ be a section,
and let $\phi:M^\vee \otimes M\rightarrow E\otimes M$ be the associated morphism.  Then
$$\cup^\vee s =H^1(\phi):H^1(M^\vee \otimes M)\rightarrow H^1(E\otimes M).$$
     \end{lem}

\begin{proof} Let $$ 0 \rightarrow M\rightarrow \mathscr M \rightarrow M \rightarrow 0 $$
 be a first order deformation of $M$, corresponding to the cocycle $\alpha\in H^1(M^\vee \otimes M)$.
 Then there is a commutative diagram
  $$ \begin{CD} 0 @.0 @.\\ @VVV @VVV\\  M^\vee \otimes M @>\phi>> E\otimes M  \\
   @VVV @VVV\\ M^\vee \otimes \mathscr{M} @>>> E\otimes \mathscr{M} \\  @VVV @VVV\\
    M^\vee\otimes M @>\phi>> E\otimes M\\  @VVV @VVV\\ 0 @.0. @. \\ \end{CD}$$
    Taking global sections we get
     $$ \begin{CD}  H^0(M^\vee \otimes M) @>H^0(\phi)>> H^0(E\otimes M)  \\ @VVV @VVV\\
      H^1(M^\vee\otimes M) @>H^1(\phi)>> H^1(E\otimes M).\\ \end{CD} $$
       By definition, $1_M\mapsto \alpha\in H^1(M^\vee\otimes M)$.
        On the other hand,  we have observed above that $H^0(\phi)(1_M)=s$, and thus $\alpha \cup^\vee s=H^1(\phi)(\alpha)$.
         \end{proof}


Let $s\in H^0(E\otimes M)$ be a section, and $M^\vee \to E$ be the associated morphism.
 Then there is a commutative diagram
\begin{equation}\label{eqnextension} \begin{CD} @. M^\vee @>>> E@>>>E/M^\vee@>>>0\\ @. @VVV @| @VVV\\ 0 @>>> L @>>> E  @>>>F @>>>0\\ \end{CD} \end{equation}
 where $L$ and $F$ are vector bundles, and  $L/M^\vee$ is torsion. Tensoring by $M$, and taking cohomology,
  we get a diagram $$ \begin{CD} H^1(M^\vee\otimes M) @>\cup^\vee s>> H^1(E\otimes M)\\ @VVV @| \\
   H^1(L\otimes M) @>>> H^1(E \otimes M) @>>>H^1(F \otimes M)@>>>0\\ @VVV \\ 0\\ \end{CD} $$
   where the surjection on the left is due to the fact that $L/M^\vee$ is torsion.
    It is then clear that $\cup^\vee s:H^1(M^\vee \otimes M)\rightarrow H^1(E\otimes M)$
    is surjective if and only if $h^1(F\otimes M)=0$. Taking duals, we get

$$ \begin{CD} H^0(M^\vee\otimes M\otimes K_C) @<\cup^\vee s<< H^0(E^\vee \otimes M^\vee\otimes K_C)\\
 @AAA @| \\ H^0(L^\vee\otimes M^\vee\otimes K_C)
  @<<< H^0(E^\vee \otimes M^\vee \otimes K_C) @<<<H^0(F^\vee \otimes M^\vee \otimes K_C)\\ @AAA \\ 0\\ \end{CD},$$
where the map on the bottom right is an inclusion.
This proves the following lemma:

\begin{lem}\label{lem35} In the notation above,
$$\cup^\vee s :H^0(E^\vee \otimes M^\vee \otimes K_C)\rightarrow H^0(M^\vee \otimes M \otimes K_C)$$
is injective if and only if $H^0(F^\vee \otimes M^\vee \otimes
K_C)=0$,  and is zero if and only if $H^0(F^\vee \otimes M^\vee
\otimes K_C)=H^0(E^\vee \otimes M^\vee \otimes K_C)$.
     \end{lem}
\begin{proof}
This follows from the discussion above.
\end{proof}

\begin{pro}\label{proh02}
Let $E$, $M$, and $\Theta_E$ be as above.  If $h^0(E\otimes M)\le
2$, then  $\textnormal{mult}_M\Theta_E>h^0(E\otimes M)$ if and only
if there is an extension of vector bundles
$$
0\rightarrow L \rightarrow E \rightarrow F \rightarrow 0,
$$
such that $h^0(L\otimes M)+h^1(F\otimes M)>h^0(E\otimes M)$.
     \end{pro}
\begin{proof}  If there exists such an extension, then the conclusion of the lemma is a special case of Lemma
\ref{lemfirst}.  Conversely, if
$\textnormal{mult}_M\Theta_E>h^0(E\otimes M)$, we use the discussion
above to generalize the arguments in Laszlo \cite{l}.

Assume now $h^0(E\otimes M)=2$.  The case $h^0(E\otimes M)=1$ is easier, and can also be proven using the same argument.
 Consider the dual cup product map
 $$ H^0(E\otimes M)\otimes H^0(E^\vee \otimes M^\vee \otimes K_C)\rightarrow H^0(M\otimes M^\vee \otimes K_C), $$
 and, after choosing bases,  let $A$ be the associated two by two matrix, with entries in
  $H^0(M\otimes M^\vee \otimes K_C)=H^1(M\otimes M^\vee)^\vee$. We have seen in Lemma
  \ref{lemcupdet} that  $\textnormal{mult}_M\Theta_E>h^0(E\otimes M)$ if and only if the determinant of $A$
  is identically zero.

 It is a result of Eisenbud and Harris \cite{eh} that this determinant is zero if and only if there exist bases for
  $H^0(E\otimes M)$ and $ H^0(E^\vee\otimes M^\vee \otimes K_C)$ such that either a column or row of $A$ is zero.
   This implies that there is either a section $s\in H^0(E\otimes M)$ such that
   $\cup^\vee s : H^0(E^\vee\otimes M^\vee \otimes K_C) \rightarrow H^0(M\otimes M^\vee \otimes K_C)$
   is the zero map, or a section  $t\in H^0(E^\vee\otimes M^\vee\otimes K_C)$ such that
    $\cup^\vee t:  H^0(E\otimes M) \rightarrow H^0(M\otimes M^\vee \otimes K_C)$ is the zero map.

 In the former case, it follows from  Lemma \ref{lem35} that there is an extension constructed as above in (\ref{eqnextension}) such that $h^1(E\otimes M)=h^1(F\otimes M)$.
  Moreover, the image of $1_M\in H^0(M^\vee \otimes M)$ in $H^0(L\otimes M)$ is nonzero, so that  $h^0(L\otimes M)\ne 0$
  and thus $E$ has an extension as in the statement of the lemma.
The same argument shows that in the latter case, $E^\vee \otimes K_C$  has an extension as in the statement of the
lemma (with $M$ replaced by $M^\vee$); by Serre duality, these conditions are equivalent. \end{proof}

\begin{rem}
Theorem \ref{teo2} is a direct consequence of Proposition \ref{proh02} and Lemma \ref{lemfirst}.
\end{rem}

\subsection{Theta divisors in $\textnormal{Pic}^0(C)$}\label{sectd}
For generalized theta divisors on $\textnormal{Pic}^0(C)$, we can
make some higher order computations, which in  the case of $2\Theta$
divisors yield some generic results. Suppose that $E\in
\mathscr{U}(r,r(g-1))$ and $\xi\in \textnormal{Pic}^0(C)$. From the
previous section,
$$
\textnormal{mult}_\xi \Theta_E \ge h^0(E\otimes \xi),
$$
and $\textnormal{mult}_\xi \Theta_E>h^0(E\otimes \xi)$ if there exists an extension of vector bundles
$$
0\rightarrow L \rightarrow E\otimes \xi \rightarrow F\rightarrow 0
$$
such that $h^0(L)+h^1(F)>h^0(E\otimes \xi)$.  Moreover, if
$h^0(E\otimes \xi)\le 2$, then  $\textnormal{mult}_\xi
\Theta_E>h^0(E\otimes \xi)$ if and only if such an extension exists.

In the rank two case, we can say more due to the fact that $L$ and
$F$ are line bundles.  In fact, the properties  of $L$ and $F$ that
we will need may hold for generic vector bundles of higher rank, and
so we give these conditions a name.  Roughly speaking, we will say
that a vector bundle is a \emph{first order deformation vector
bundle} if it behaves like a line bundle.  Recall that for a vector
bundle $E$ on $C$, and $\alpha\in H^1(\mathscr O_C)$, we set
$K_{E,\alpha}=\ker \left( \cup \alpha:H^0(E)\to H^1(E)\right)$.

\begin{dfn}\label{dfnfodvb}
Let $E$ be a vector bundle of slope $\mu(E)=\mu$ on a smooth curve $C$ of genus $g\ge 2$.
\begin{enumerate}
\item If $\mu(E)=g-1$, then $E$ is a \emph{first order deformation vector bundle} if there exists an
 $\alpha \in H^1(\mathscr O_C)$ such that $\cup \alpha:H^0(E)\to H^1(E)$ is injective.

\item If $\mu(E)<g-1$, then $E$ is a \emph{first order deformation vector bundle} if $h^0(E)>0$,
 and given $W\subseteq H^1(E)$, with
$$
h^0(E)+\dim W\le h^1(E),
$$
there exists $\alpha\in H^1(\mathscr O_C)$ such that
$\cup \alpha:H^0(E)\to H^1(E)$ is injective and
$(\alpha\cup H^0(E))\cap W=0$.

\item If $\mu(E)>g-1$ then $E$ is a \emph{first order deformation vector bundle} if $h^1(E)>0$
 and given $W^\perp\subseteq H^0(E)$, with
$$
\dim W^\perp\le \chi (E),
$$
there exists $\alpha\in H^1(\mathscr O_C)$ such that
$\dim K_{E,\alpha}=\chi(E)$ (i.e. $\cup \alpha$ is surjective), and
$K_{E,\alpha}\cap W^\perp=0$.
\end{enumerate}
\end{dfn}

\begin{rem}
It is easy to check that if $E$ is a first order deformation vector
bundle, then $K_C\otimes E^\vee$ is as well.
\end{rem}

\begin{lem} \label{lemfodvb}
A line bundle $E$ on $C$ is a first order deformation vector bundle.
     \end{lem}

\begin{proof}
As (2) and (3) are dual to each other,  it is enough to prove (2).
This can be done using Lemma \ref{lemcomp},  in the special case of
a general effective divisor $D$ of degree $h^1(E)$. The details are
left to the reader. The proof of (1) is similar.  See for example \cite[Theorem 1.9]{cmf} where this computation is done in detail.
\end{proof}

\begin{lem}\label{lemfirstorder}
Suppose that $E\in \mathscr{U}(r,r(g-1))$, $\xi\in
\textnormal{Pic}^0(C)$, and there exists an extension of  vector
bundles
$$
0\rightarrow L \rightarrow E\otimes \xi \rightarrow F\rightarrow 0,
$$
such that $h^0(L)+h^1(F) \ge h^0(E\otimes \xi)$.  If $L$ and $F$ are
first order deformation vector bundles, then for a general
$\alpha\in H^1(\mathscr O_C)$, corresponding to a first order
deformation $\mathscr E'$ of $E\otimes \xi$, it follows that
$h^0(\mathscr E')=h^0(L)+h^1(F)$.  In particular, this is true when
$E$ has rank two.
     \end{lem}
\begin{proof}  Set $E'=E\otimes \xi$.
Since $E$ is stable, $L$ is a vector bundle of slope less than $g-1$.
Thus taking $\alpha$ general,  $K_{L,\alpha}=0$.  The proof of Lemma
\ref{lemfirst} shows that in order to prove the lemma above, it
suffices to establish that $\dim(\cup \alpha H^0(L)) \cap (\cup^\vee
e H^0(F))=h^0(L)+h^1(F)-h^0(E')$, and $(H^0(E')/H^0(L)) \cap
K_{F,\alpha}=0$.  These two statements follow from the definition of
first order deformation vector bundle and the observation that
$h^0(L)+(h^0(F)-h^0(E')+h^0(L))-h^1(L)=h^0(L)+h^1(F)-h^0(E')$, and
$h^0(E')-h^0(L)\le h^1(F)$.
\end{proof}

We now turn our attention to second order deformations.  We have the following technical lemma:

\begin{lem}\label{lemtec}
Suppose that $E\in \mathscr{U}(r,r(g-1))$, $\xi\in \textnormal{Pic}^0(C)$, and there exists an extension of vector bundles
$$
0\rightarrow L \rightarrow E\otimes \xi \rightarrow F\rightarrow 0,
$$
corresponding to $e\in Ext^1(F,L)$ such that
$h^0(L)+h^1(F)>h^0(E\otimes \xi)$,  and $L$ and $F$ are first order deformation vector bundles.
If there exists $\alpha\in H^1(\mathscr{O}_C)$, such that
$(\cup^\vee e K_{F,\alpha}) \cap (\cup^\vee \alpha H^0(L))=0\subseteq H^1(L)$, then
$\textnormal{mult}_\xi \Theta_E=h^0(L)+h^1(F)$.
In particular, this is true if $E$ has rank two.
     \end{lem}

\begin{proof}
The key observation is from the proof of the previous lemma, where
we showed that if $\alpha\in H^1(\mathscr{O}_C)$  is general, then
$(H^0(E\otimes \xi)/H^0(L)) \cap K_{F,\alpha}=0$. In other words,
for a general $\alpha\in H^1(\mathscr{O}_C)$, the sections of
$E\otimes \xi$ which lift to first order are a subspace of
$H^0(L)\subseteq H^0(E\otimes \xi)$. To prove the lemma, we must
show that nontrivial sections of $E\otimes \xi$ do not lift to
second order for a general second order deformation (Lemmas
\ref{lem19}, \ref{lem18}, and \ref{lem11}). So let $\mathscr{M}_2$
be a general  second order deformation of the trivial bundle:
$$
0\rightarrow \mathscr{O}_C\rightarrow \mathscr{M}_2\rightarrow \mathscr{M}_1 \rightarrow 0.
$$

Again setting $E\otimes \xi=E'$, we have a commutative diagram similar to (\ref{eqnbig})
which gives rise to a diagram of long exact sequences, of which the following is a part

$$
\begin{CD}
0@>>> H^0(L\otimes \mathscr{M}_2) @>>> H^0(E'\otimes\mathscr{M}_2) @>>>H^0(F \otimes \mathscr{M}_2) \\
@. @VVV @VVV @VVV\\
0@>>> H^0(L\otimes \mathscr{M}_1) @>>> H^0(E'\otimes\mathscr{M}_1) @>>>H^0(F \otimes \mathscr{M}_1) \\
@. @VVV @VVV @VVV\\
 @>>> H^1(L) @>>> H^1 (E') @>>>H^1(F). \\
\end{CD}
$$

It follows that if a section of $H^0(E'\otimes \mathscr{M}_1)$ lifts
to second order, then its image  in $H^0(F\otimes \mathscr{M}_1)$
lifts to second order as well. Now, as we observed at the beginning,
a first order lift of a section of $E'$ maps to a first order lift
of the zero section in $H^0(F)$. A first order lift of the zero
section of $H^0(F)$ is just a section of $F$ (times $t$, where $t$
is a local parameter for the deformation), and thus will lift to
second order only if it is in the space of sections of $H^0(F)$
which lift to first order. In other words,  we have an induced map
$\phi_\alpha: H^0(E'\otimes \mathscr{M}_1)\rightarrow t\cdot
H^0(F)$. We also have the space $K_{F,\alpha}=\ker(\cup^\vee
\alpha:H^0(F)\rightarrow H^1(F))$, and we would like to give a
condition for $\operatorname{im}(\phi_\alpha)\cap K_{F,\alpha}=0$.

Extending diagram (\ref{eqnsma}) to the upper right
we have
$$
\begin{CD}
 @.  @. H^0(L)\\
 @. @. @V\cup^\vee \alpha VV\\
 H^0 (E') @>>>H^0(F) @>\cup^\vee e >>H^1(L)\\
@V\cdot t VV @V\cdot t VV @VVV\\
 H^0(E'\otimes \mathscr{M}_1) @>>>H^0(F\otimes \mathscr{M}_1) @>>> H^1(L\otimes \mathscr{M}_1)\\
@VVV @VVV @VVV\\
 H^0 (E') @>>>H^0(F) @>>> H^1(L)\\
\end{CD}
$$
and we see that $\cup^\vee e \circ \phi_\alpha H^0(\mathscr{E}') \subseteq \cup^\vee \alpha H^0(L)$.

Again using the fact that
$(H^0(E')/H^0(L)) \cap K_{F,\alpha}=0$, and
chasing through the top row of the diagram, we see that if $\cup^\vee e K_{F,\alpha}$ intersects trivially with
$\cup^\vee \alpha H^0(L)$, then
$\operatorname{im}(\phi_\alpha)\cap K_{F,\alpha}=0$.  Consequently, nontrivial sections of $H^0(E')$ can not lift to second order, completing the proof.
\end{proof}

We will use this to find the multiplicity in the general situation
for rank two vector bundles.   The idea is that the conditions of
the lemma should hold for a general such extension. In order to
state the results, we will want to introduce a parameter space of
vector bundles, so that we can make sense of the term ``general''.
So define a set
$$
B\mathscr U(2,2(g-1)):=
\bigcup_{0\le d<g-1}
\{(L,F,e): L\in B^1_d, \ K_C\otimes F^\vee \in B^1_d,\ e \in
H^1(F^\vee \otimes L)\}.
$$

A result of Laszlo, \cite{l} Proposition V.4, states that given
$(L,F,e)\in B\mathscr U$,
 if $e$ is general, then the extension $E$ corresponding to $e$ is stable, and,
either $h^0(L)=h^0(E)$ or $h^1(E)=h^1(F)$, so that
$h^0(L)+h^1(F)> h^0(E)>0$.  Let $\mathscr B\mathscr U$ be the subset with the property that
$h^0(L)+h^1(F)>h^0(E)>0$.

Thus for each point in $\mathscr {BU}$ we get a vector bundle together with certain extension data.  Now suppose conversely that
$E\in \mathscr{U}(2,2(g-1))$, and there exists an extension of vector bundles
$$
0\rightarrow L \rightarrow E \rightarrow F\rightarrow 0,
$$
such that $h^0(L)+h^1(F) > h^0(E)>0$.  Since both $h^0(L)$ and
$h^1(F)$ are at most equal to $h^0(E)$,  it is clear that both are
nonzero.
Thus $\mathscr B\mathscr U$ is equal to the set of stable
vector bundles $E$ together with an extension
$$
0\rightarrow L \rightarrow E \rightarrow F\rightarrow 0,
$$
with $L\in B^1_{1,d},\ K\otimes F^{\vee}\in B^1_{1,d},$ such that
$h^0(L)+h^1(F)> h^0(E)>0$.

We can check in fact that the general vector bundle $E$ admitting such an extension comes from a unique
such extension: Assume the opposite, namely that for such an
extension, there exists another extension
$$
0\rightarrow L' \rightarrow E \rightarrow F'\rightarrow 0,
$$
with say $d':=\deg (L')\ge \deg (L)=d$. The map $L'\rightarrow E$ cannot
factor through $L$ and therefore gives rise to a non-zero map
$L'\rightarrow F$. We then have a commutative diagram
$$
\begin{matrix}
0& \rightarrow &L &\rightarrow& E\times_F L' &\rightarrow &L' &\rightarrow&0 \\
& & \downarrow& &\downarrow & &\downarrow&  & \\
0&\rightarrow& L &\rightarrow& E &\rightarrow&F&\rightarrow&0 \\
\end{matrix}
$$
 where the top row is obtained by pull-back. The map $L'\rightarrow F$
 identifies the image of $L'$ to $F(-D)$. As it
 comes from a map $L'\rightarrow E$, the top row is split. Hence, the map
 $$\psi :H^1(F^{\vee}\otimes L)\rightarrow H^1(L^{'\vee}\otimes L)=H^1(F^{'\vee}\otimes L(D))$$ maps the
 extension $e$ corresponding to  the second row in the diagram above to zero.
 The dimension of the kernel of $\psi$ is $h^1(F^{\vee}\otimes L)-h^1(L^{'\vee}\otimes
 L)$. The dimension of the set of (effective) $L'$ inside $F$ is
a projective space of dimension $h^0(F)-1$. Each one of them
contains a finite number of effective divisors of a given degree. If
a generic $E$ is to be obtained in such a
 way, we must have
 $$h^0(F)-1+h^1(F^{\vee}\otimes L)-h^1(L^{'\vee}\otimes L)\ge h^1(F^{\vee}\otimes
 L)$$

 This gives
 $d'-d+g-1\le h^1(L^{'\vee}\otimes L)\le h^0(F)-1$. As $d'\ge d$ by
 assumption and $d+degF=2g-2$, the above is only possible if $d'=d,
 F=K$ and hence $L'=L=\mathcal O_C$.  Then the inequality becomes
  $h^1(\mathcal O_C)\le h^0(K)-1$ which is also false.

\bigskip

By ``a general vector bundle in $\mathscr
B\mathscr U$'' we will mean a vector bundle $E\in \mathscr
U(2,2g-2)$  such that there exist data $L$, $F$, and $e$ as above  where $e\in
H^1(F^\vee \otimes L)$ is the class of the extension, and all the
data are general with the given conditions. We point out that these
``general'' vector bundles move in subsets of dimension $3g-4 $, and
there is one such subset for every $d$ with $0\le d\le g-2$ where
$d$ is the degree of $L$.

\begin{lem}\label{Egen}
Suppose that $E\in \mathscr B\mathscr U$ is ``general'' and
$$
0\rightarrow L \rightarrow E \rightarrow F\rightarrow 0
$$
is an extension
such that $h^0(L)+h^1(F)>h^0(E)$.  Then
$\textnormal{mult}_{\mathscr O_C} \Theta_E=h^0(L)+h^1(F)$.
     \end{lem}
\begin{proof}
We must check that for general $\alpha\in H^1(\mathscr O_C)$, and
for general $e\in H^1(F^\vee \otimes L)$,  that $\cup^\vee e
K_{F,\alpha}\cap \cup^\vee \alpha H^0(L)=0\subseteq H^1(L)$.  After
fixing $\alpha$, and observing that $\dim K_{F,\alpha}+\dim
h^0(L)\le h^1(L)$, it suffices to check that by taking $e$ general,
the above condition is satisfied.  This is straightforward to check
using Lemma \ref{lemfodvb}.
 \end{proof}

\begin{proof}[Proof of Theorem \ref{pro}]
By the result of Laszlo mentioned earlier, the general element of the irreducible variety $B\mathscr U(2,2(g-1))$ is stable and the open subset $\mathscr B\mathscr U$ given by the condition $h^0(L)+h^1(F)>h^0(E)>0$ is non-empty.  The theorem now follows from Lemma \ref{Egen}.
\end{proof}

\section{Existence of exceptional vector bundles}\label{secexc}
We will say that a vector bundle $E$ is \emph{exceptional} if there exists
a line bundle  $\xi\in \textnormal{Pic}^0(C)$ and an extension
$$
0\to E' \to E\otimes \xi \to E'' \to 0
$$
such that $h^0(E')+h^1(E'')>h^0(E\otimes \xi)$.
Observe that if $E$ is exceptional, then so is $E^\vee\otimes K_C$.
In this section we show
\begin{teo}\label{teoexc}
Let $C$ be a smooth curve of genus $g\ge 2$.  For any integer $r>1$, there exist
exceptional stable vector bundles $E$ of rank $r$ and  slope $g-1$, such that $\Theta_E$ is a divisor.
 In other words, on $\textnormal{Pic}^0(C)$, there exist generalized theta divisors
  $\Theta_E\in |r\Theta|$, and line bundles $\xi \in \textnormal{Pic}^0(C)$,  such that
$mult_\xi\Theta_E>h^0(E \otimes \xi)$.
On the other hand,  the generic vector bundle in $\mathscr U(r,r(g-1))$ is not exceptional.     \end{teo}
\begin{proof}
This follows from the propositions below.
\end{proof}

\begin{rem}
 The existence of exceptional rank two vector bundles was proven by Laszlo \cite{l}, while
a result of the second author \cite{t} shows that a general vector bundle of rank $2$ and slope $g-1$ is not exceptional.
\end{rem}

\begin{subsection}{Existence of exceptional vector bundles}
We want to see that there exist vector bundles $E$ of any rank $r$
and slope $g-1$ fitting in an exact sequence
$$0\rightarrow E'\rightarrow E\rightarrow E''\rightarrow 0$$ such that
$E$ is stable of slope $g-1$, $E$ defines a proper theta divisor,
and $h^0(E)<h^0(E')+h^1(E'')$. As the case $r=2$ has been dealt with
already, we can assume $r\ge 3$. We are going to construct such $E$
by taking $E'$ a generic vector bundle of rank $r-1$ and degree
$(r-1)(g-1)-1$ with one section, $E''$ a generic line bundle of
degree $g$ and $h^1(E'')=1$, and $E$ a generic extension of these
vector bundles.

\begin{pro}\label{pro43}
Let $C$ be a curve of genus at least two. Let $E'$ be a generic
vector bundle of rank $r-1$ and degree $(r-1)(g-1)-1$ with
$h^0(E')=1$, let $E''$ be a generic line bundle of degree $g$ and
$h^1(E'')=1$. For a general extension
$$
0\to E' \to E  \to E''\to 0,
$$
$E$ is stable, $h^0(E)=1$, and $\Theta_E\subset \textnormal{Pic}^0(C)$ is a divisor.
\end{pro}

\begin{proof}
We check each of the conditions separately, namely that  $E$ is
stable, $h^0(E)=1$ and $E$ defines a genuine theta divisor.

\begin{proof}[Proof of stability]

Assume that $E$ were not stable. It would then have a subbundle $F$
of slope at least $g-1$. Consider the following diagram
\begin{equation*}
\begin{matrix}0&\rightarrow &E'&\rightarrow &E&\rightarrow &E''&\rightarrow &0\\
 & &\uparrow & &\uparrow & &\uparrow& & \\
0&\rightarrow &F'&\rightarrow &F&\rightarrow &F''&\rightarrow &0\\
\end{matrix}
\end{equation*}
where $F''$ is the image of $F$ in $E''$. If $F''=0$, then $F$ is a
subbundle of $E'$. As $E'$ is stable of slope smaller than $g-1$,
$F$ cannot contradict the stability of $E$. Hence $F''$ is a line
bundle. If $\deg F''\le g-1$, again $F$ cannot contradict the
stability of $E$. So we can assume that $F''=E''$ has degree $g$.
Then $F'\not= 0$, since otherwise the original sequence would split.

As $E'$ is a generic vector bundle of rank $r-1$ and degree
$(r-1)(g-1)-1$ with one section, by virtue of Theorem IV.2.1 \cite{sundaram} it can be written in an exact
sequence
$$0\rightarrow {\mathcal O_C}\rightarrow E'\rightarrow \bar E\rightarrow
0$$
where $\bar E$ is a generic vector bundle of rank $r-2$ and degree
$(r-1)(g-1)-1$. Let $\bar F$ be the image of $F'$ in $\bar E$. We
have an exact diagram
\begin{equation}\label{eqnexact}
\begin{matrix}0&\rightarrow &{\mathcal O_C}&\rightarrow &E'&\rightarrow &\bar E&\rightarrow &0\\
 & &\uparrow & &\uparrow & &\uparrow& & \\
0&\rightarrow &L&\rightarrow &F'&\rightarrow &\bar F&\rightarrow &0.\\
\end{matrix}
\end{equation}
As $\operatorname{rk} F<\operatorname{rk} E$,     $\operatorname{rk}
F''=\operatorname{rk}E''$, it follows that  $\operatorname{rk}
F'<\operatorname{rk} E'=r-1$.

If $\operatorname{rk}\bar F=\operatorname{rk}\bar E$, then $L=0$ as
$F'$ is a proper subbundle of $E'$. We deal with this case later and
assume for now $L\not= 0$. Replacing if necessary $\bar F$ by the
subbundle of $\bar E$ that it generates, we have by the genericity
of $\bar E$ that the difference of slopes between $\bar E/\bar F$
and $\bar F$ is at least $g-1$ (\cite{RT} Th 0.1.). This can be
written as

$$\frac{(r-1)(g-1)-1-d_{\bar F}}{(r-1)-1-r_{\bar F}}-\frac{d_{\bar F}}{r_{\bar F}}\ge g-1,$$
which translates into
$$d_{\bar F}\le \frac{r_{\bar F}}{r-2}((r_{\bar F}+1)(g-1)-1).$$
As $L\not= 0$ is a subsheaf of ${\mathcal O}_C$,  $\deg(L)\le 0$.
This gives
$$d_{F'}\le d_{\bar F}\le \frac{r_{F'}-1}{r-2}((r_{ F'}(g-1)-1).$$
Then, $$d_F\le g+\frac{r_{F'}-1}{r-2}((r_{ F'}(g-1)-1).$$
We want to
check that the slope of $F$ is smaller than $g-1$. It suffices then
to see that
$$g+\frac{r_{F'}-1}{r-2}(r_{ F'}(g-1)-1)<(r_{F'}+1)(g-1)$$
or equivalently $$0<(r-r_{F'}-1)(r_{F'}(g-1)-1),$$ which holds as
$g\ge 2, \ r_{F'}<r_F<r$ except when $g=2, r_{F'}=1$. We deal with
this exceptional case later.

Assume now $L=0$ and hence $\bar F=F'$. As $\bar E$ is generic, the
set of subsheaves $ F'$ of $\bar E$ has dimension (see \cite{RT}
Theorem 0.2.)
$$x= r_{F'}d_{\bar E}-r_{\bar E}d_{F'}-(r_{\bar E}-r_{F'})r_{F'}(g-1).$$

If such an $\bar F$ lifts to $E'$, the extension corresponding to
$E'$ is in the kernel of the map $H^1(\bar E^\vee)\rightarrow
H^1(F^{' \vee})$. If this happens for the generic extension, one must
have
$$x+h^1(\bar E^\vee)-h^1(F^{' \vee})\ge h^1(\bar E^\vee).$$

If $r_{F'}=1$, then $d_{F'}=g-2$ and therefore $h^0(F^{' \vee})=0$
or $g=2, F'={\mathcal O}_C$. We deal with the case $r_{F'}=1, \ g=2$
later. Hence, if $h^0(F^{' \vee})>0$, we can assume $r_{F'}>1$.
Moreover, there is a non-zero map $F'\rightarrow {\mathcal O_C}$.
The kernel $A$ of this map has rank $r_{F'}-1$ and degree at least
equal to the degree of $F'$ and is a subsheaf of $E$. If $F$
contradicts the stability of $E$ a little bit of arithmetic can be
used to check that so does $A$. Hence, repeating the process if
needed, we can assume $h^0(F^{' \vee})=0$.

If $h^0(F^{' \vee})=0$, then $h^1(F^{' \vee})=d_{F'}+r_{F'}(g-1)$.
Using that $r_{\bar E}=r-2, d_{\bar E}=(r-1)(g-1)-1$, the inequality
above implies
$$d_{F'}\le \frac{r_{F'}(r_{F'}(g-1)-1)}{r-1}.$$
Then
$$d_F=g+d_{F'}\le g+ \frac{r_{F'}(r_{F'}(g-1)-1)}{r-1}.$$
It suffices to check that the latter is smaller than
$(r_{F'}+1)(g-1)$. This is equivalent to $$(g-1)r_{F'}>1$$ which is
true except if $g=2, r_{F'}=1$.

Assume now $g=2, r_{F'}=1$. As $F''=E''$ has rank one and degree
$g=2$, if $F$ contradicts stability of $E$, $F'$ has degree at least
zero. As $\bar E$ is generic of rank and degree $r-2$, it has a
one-dimensional family of line bundles of degree zero and no
subbundles of degree one (see \cite{RT} Th0.2.). As either
$F'={\mathcal O}_C$ or $F'$ gives rise to a subbundle of $\bar E$,
$F'$ has degree zero.

Let us check that $E'$ has only a finite number of subbundles of
degree zero. If $L$ is a line subbundle of $\bar E$ that lifts to
$E'$, there is then an exact diagram

\begin{equation*}
\begin{matrix}0&\rightarrow &{\mathcal O}_C&\rightarrow &E'&\rightarrow &\bar E&\rightarrow &0\\
 & &\uparrow & &\uparrow & &\uparrow& & \\
0&\rightarrow &{\mathcal O}_C&\rightarrow &E'\times _{\bar E}L&\rightarrow &L&\rightarrow &0\\
\end{matrix}
\end{equation*}
and therefore in the map  $\varphi :\ H^1(\bar E^{\vee})\to H^1( L^{
\vee})$ the extension class corresponding to the first row maps to
zero. As $\varphi$ is onto, $H^1( L^{ \vee})\not= 0$ and by
assumption $e$ is a generic element in $H^1(\bar E^{ \vee})$, the
generic $L$ will not lift to $E'$ and $E'$ can only have a finite
number of line subbundles $F'$ of degree zero.

From the exact diagram \begin{equation*}
\begin{matrix}0&\rightarrow &E'&\rightarrow &E&\rightarrow &E''&\rightarrow &0\\
 & &\uparrow & &\uparrow & &\uparrow& & \\
0&\rightarrow &F'&\rightarrow &F&\rightarrow &E''&\rightarrow &0\\
\end{matrix}
\end{equation*}
the extension corresponding to $E$ is in the image of $H^1( E^{''
\vee}\otimes F')\to H^1( E^{'' \vee}\otimes E')$. The cokernel of
this map is $H^1( E^{'' \vee}\otimes (E'/F'))\not= 0$. Hence, the
map is not onto and the generic extension is not in the image. As
there are only a finite number of possible $F'$, this concludes the
proof.
\end{proof}

\begin{proof}[Proof that the number of sections is one]
By assumption, $E'$ has only one section. Hence, if $E$ had more
than one section, it would mean that there is a section of $E''$ that
lifts to $E$. Hence, there is an exact diagram
$$\begin{matrix}0&\rightarrow &E'&\rightarrow &E&\rightarrow &E''&\rightarrow &0\\
 & &\uparrow & &\uparrow & &\uparrow& & \\
0&\rightarrow &E'&\rightarrow &E\times_{E''}{\mathcal O_C}&\rightarrow &{\mathcal O_C}&\rightarrow &0.\\
\end{matrix}$$
This implies that the extension corresponding to $E$ is in the
kernel of the natural map $H^1(E^{'' \vee}\otimes E')\rightarrow H^1(E').$

As $E''$ is a line bundle of degree $g$, the cokernel of the map
$E^{'' \vee}\otimes E'\rightarrow E'$ is torsion (of degree $(r-1)g$).
Hence, the induced map
$$
H^1(E^{'' \vee}\otimes E')\rightarrow H^1(E')
$$
is onto.
As $h^0(E')=1$, $\deg(E')=(r-1)(g-1)-1$, by Riemann-Roch $h^1(E')=2$.
As $E''$ is a  line bundle of degree $g$ with $h^1(E'')=1$,
$h^0(E'')=2$. Note that two sections of $E''$ that differ in
multiplication by a constant would give rise to the same pull-back
extension in the diagram above. If for the generic extension, at
least one section lifts to $E$, we have
$$h^0(E'')-1+h^1(E^{''\vee}\otimes E')-h^1(E')\ge h^1(E^{'' \vee}\otimes E'),$$
which is false from the numbers we just computed.
\end{proof}

\begin{proof}[Proof that the generic such $E$ defines a proper theta divisor]
We must check that  there exist line bundles of degree zero such
that $h^0(E\otimes L)=0$.
We first check that for a generic $L$ of degree zero, $h^0(E'\otimes
L)=0$. Tensoring the exact sequence defining $E'$ with $L$, we
obtain
$$0\rightarrow L\rightarrow E'\otimes L\rightarrow \bar E\otimes
L\rightarrow 0.$$
As $\bar E$ is generic, so is $\bar E\otimes L$.
Hence, $h^0(\bar E\otimes L)=g-2$. If a section of $\bar E\otimes L$
lifts to a section of $E'\otimes L$, an argument as above says that
the corresponding extension is in the kernel of the natural map
$H^1(\bar E^\vee)\rightarrow H^1(L)$. If this happens for the
generic extension,
$$h^0(\bar E\otimes L)-1+h^1(\bar E^{\vee})-h^1(L)\ge h^1(\bar
E^{\vee}).$$
As $L$ is generic of degree zero, $h^1(L)=g-1$. By the
genericity of $\bar E$, $h^0(\bar E\otimes L)=g-2$. Then, it is easy
to check that  the inequality above is not satisfied.

Assume now $h^0(E\otimes L)=1$. Then, this section comes from a
section of $E''\otimes L$. By the genericity of $L$, $h^0(E''\otimes
L)=1$ and is generated by a section, say $s$. Hence, this unique
section of $E''\otimes L$ must lift to every extension of $E''$ by
$E'$. Equivalently, the map
$$H^1(E^{''\vee}\otimes E')\rightarrow H^1(L\otimes E')$$
induced by $s$ is identically zero, which is false, as the map
$E^{''\vee}\otimes E'\rightarrow L\otimes E'$ has torsion cokernel and therefore the map induced in cohomology is surjective.
\end{proof}
This completes the proof of Proposition \ref{pro43}.
\end{proof}

\end{subsection}

\subsection{The generic case}
We would like now to consider the question of whether a general vector bundle in $\mathscr U(r,r(g-1))$ is exceptional.

\begin{pro}\label{generic}
Given a generic vector bundle $E$ of rank $r$ and degree $r(g-1)$,
there do not exist stable vector bundles $E',E''$ with $E$ fitting
in an exact sequence
$$0\rightarrow E'\rightarrow E\rightarrow E''\rightarrow 0$$ and such that
there exists a line bundle $\xi$ of degree zero with
\begin{equation}\tag{$*$}\label{eqnstar}
h^0(E\otimes \xi)<  h^0(E'\otimes \xi)+h^1(E''\otimes \xi).
\end{equation}
In other words, the generic vector bundle is not exceptional.
\end{pro}

\begin{proof}
 Note that
$$h^0(E'\otimes \xi)\le h^0(E\otimes \xi)=h^1(E\otimes \xi),
h^1(E''\otimes \xi)\le h^1(E\otimes \xi).$$ Hence, if
\eqref{eqnstar} is correct, both $h^0(E'\otimes \xi)\not= 0,\
h^1(E''\otimes \xi)\not= 0$.

We assume first that $\xi=\mathcal O_C$.
 As
$h^0(E')\ge 1$, $E$ has a section and therefore there is a non-zero
map $\mathcal O_C\rightarrow E$. Let $A$ be the divisor of zeroes of
this map and $d_A$ its degree. Then we have an exact sequence
$$0\rightarrow \mathcal O_C(A)\rightarrow E\rightarrow \bar E
\rightarrow 0.$$  As the map $\mathcal O_C(A)\rightarrow E$ factors
through $E'$, there is a surjection $\bar E\rightarrow E''$. Hence
$h^1(\bar E)\not= 0$. It follows that there exists a non-zero map
$\bar E\rightarrow K_C$. Let $B$ be the divisor of zeroes of its
dual and $d_{B}$ its degree. There is then an exact sequence

$$0\rightarrow \hat E \rightarrow \bar E\rightarrow K_C(-B)
\rightarrow 0.$$

{\bf Claim 1} $H^0(\bar E^\vee(A))=0$.

\begin{proof} This is a consequence of the stability of $E$
(see \cite{RT} Lemma 1.1 ).
If $H^0(\bar E^\vee(A)\not= 0$, there is a non-zero map $\bar
E\rightarrow \mathcal O_C(A)$. Composing the projection
$E\rightarrow \bar E$ with this map followed by the inclusion of
$\mathcal O_C(A)\rightarrow E$ one obtains a non-zero map from $E$
to itself that is not multiplication by a constant; this contradicts the
stability of $E$.
\end{proof}

Assume now  $H^0(\hat E \otimes K_C^\vee (B))= 0$ and count the
dimension of the vector bundles $E$ in the situation described in
Proposition \ref{generic}. As $\hat E$ has rank $r-2$, it moves in a
space of dimension at most $(r-2)^2(g-1)+1$. Thus $\bar E$ moves in
a space of dimension at most
$$x=(r-2)^2(g-1)+1+d_B+h^1(\hat E\otimes K_C^\vee(B))-1.$$
It follows that $E$ moves in a space of dimension at most
$$x+d_A+h^1(\bar
E^\vee(A))-1=(r^2-1)(g-1)-(r-2)d_A-(r-2)d_B-1.$$ If we let $\xi$
vary, we add at most $g$ dimensions to this number, but this is
still less that the dimension $r^2(g-1)+1$ of the moduli space of
vector bundles of rank $r$.

{\bf Claim 2} Let $d',r'$ be the degree and rank of $E'$. If
$H^0(\hat E \otimes K_C^\vee (B))\not= 0$, then
$$\frac{d-d'}{r-r'}<\frac{2d'}{r'}.$$

\begin{proof} Consider the diagram
$$\begin{matrix} 0&\rightarrow &\mathcal O_C(A)&\rightarrow
&E&\rightarrow &\bar E &\rightarrow 0\\
 & &\downarrow & &\downarrow & &\downarrow & & \\
 0&\rightarrow &E'&\rightarrow
&E&\rightarrow & E'' &\rightarrow 0\\
\end{matrix}$$

Denote by $\bar E'$ the cokernel of the inclusion ${\mathcal
O_C}(A)\rightarrow E'$. There is an exact sequence
$$0\rightarrow \bar E'\rightarrow \bar E \rightarrow E''\rightarrow
0.$$

 Assume now that $H^0(\hat E \otimes K_C^\vee(B))\not= 0$. There exists then
 a non-zero map $K_C(-B)\rightarrow \bar E$. The composition of this map with the projection $\bar E \rightarrow
 E''$ is zero: assume the opposite. As the map $\bar E \rightarrow K_C(-E)$ comes from a map $E''\rightarrow K_C$, there would be a nonconstant map
 $E''\rightarrow E''$ obtained from the composition
 $K_C(-B)\rightarrow E''$ and $E''\rightarrow K_C(-B)$ contradicting the
 stability of $E''$.
 Hence, there is an injective map $K_C(-B)\rightarrow \bar E'$ and we
 have an exact sequence
$$0\rightarrow K_C(-B)\rightarrow \bar E' \rightarrow \tilde E'\rightarrow
0$$

Hence $\tilde E'$ is a quotient of $E'$. As $E'$  is stable, there
is an inequality of slopes $\mu (\tilde E')>\mu(E')$. This gives
$$d_B>d_A+2(g-1)-2\frac{d'}{r'}.$$

As $K_C(-B)$ is a quotient of $E''$ which is stable, there is the
inequality of slopes $\mu (K_C(-B))>\mu (E'')$. This inequality gives
$d_B<2(g-1)-\frac{r(g-1)-d'}{r-r'}$. Combining this with the
inequality above gives the claim.
\end{proof}

{\bf Claim 3}  If $H^0(\hat E \otimes K_C^\vee (B))\not= 0$, then
$E$ is not generic.

\begin{proof}The set of vector bundles that fit in an exact sequence
$$0\rightarrow E'\rightarrow E\rightarrow E''\rightarrow 0$$ has
dimension at most (see \cite{RT} Thm 0.1.)
$$r^2(g-1)+1+r'd-rd'-r'(r-r')(g-1).$$
We want to check that this
number is smaller than $r^2(g-1)+1-g$ and therefore even when $\xi$
varies, the $E$ cannot fill the moduli space.

From Claim 2, $r'd-rd'-r'(r-r')(g-1)<(r-r')(d'-r'(g-1))$. If $d'\le
r'(g-1)-\frac{g}{r-r'}$, then the result is clear.

If $d'> r'(g-1)-\frac{g}{r-r'}$, then
$$r'd-rd'-r'(r-r')(g-1)\le -r \frac{g}{r-r'}-r'(r-r')(g-1)\le -g$$
so the result is also true in this case.
\end{proof}

This completes the proof of Proposition \ref{generic}.
\end{proof}

\section{Secant varieties}\label{secsec}

We begin by proving the following general fact:

\begin{lem}
Suppose that $E\in \mathscr{U}(r,r(g-1))$, $\xi\in \textnormal{Pic}^0(C)$, and $h^0(E\otimes \xi)=dr$, for some positive integer $d$.  If there exists an effective divisor $D$ of degree $d$ on $C$, such that
$$h^0(E\otimes \xi(-D))=h^0(E^\vee\otimes \xi^\vee \otimes K_C(-D))=0,$$
then $ \textnormal{mult}_\xi\Theta_E=h^0(E\otimes \xi)$.
     \end{lem}
\begin{proof} Set $E'=E\otimes \xi$.  If there is a $D$ as in the lemma, then a general such $D$ will suffice.  Consider the first order deformation $\mathscr E'$ of $E'$ associated to $D$ (see Lemma \ref{lemcomp}), and the exact sequence
$$
0\rightarrow H^0(E')\rightarrow H^0(E'(D))\rightarrow H^0(E'(D)|_D) \rightarrow H^1(E') \rightarrow \ldots.
$$
Using Lemmas \ref{lemcomp} and \ref{lem19} one concludes that $h^0(\mathscr E')=h^0(E')$.  The result then follows from Lemma \ref{lemtheta}.
\end{proof}

When $\Theta_E\subseteq \textnormal{Pic}^0(C)$, the tangent cones to the singularities of the theta divisor can be related to the geometry of the canonical model of the curve.  To be precise, in the case that $r=1$, $d=0$, and $d'=r'(g-1)$, so that $\Theta_E\subseteq \textnormal{Pic}^0(C)$, it follows that
for $\xi\in \textnormal{Pic}^0(C)$,
the tangent cone $\mathbb{P}C_\xi \Theta_E\subseteq
\mathbb{P}H^0(C,K_C)^\vee$.  Let $\phi_K:C\rightarrow \mathbb{P}H^0(C,K_C)^\vee$ be the canonical morphism, let $\phi_K(C)$ be the canonical model of the curve, and let $s_n(\phi_K(C))$ be the $n$-th secant variety to the curve, where we use the convention that $s_0(\phi_K(C))=\phi_K(C)$.  Then we have the following:

\begin{teo}\label{teo4}
In the notation above, suppose $E\in \mathscr U(r,r(g-1))$ and $\xi\in \textnormal{Pic}^0(C)$.  If $\Theta_E$ is a divisor,  $\operatorname{mult}_\xi \Theta_E=h^0(E\otimes \xi)$, and $h^0(E\otimes \xi)>nr$, then
$s_{n-1}(\phi_K(C))\subseteq \mathbb{P}C_\xi \Theta_E$.
\end{teo}

\begin{rem}
Observe that \cite[Theorem VI 1.6 (i)]{acgh} is the special case where $r=1$.
\end{rem}

\begin{proof} Recall from Corollary  \ref{corcup} that if $\operatorname{mult}_\xi \Theta_E=h^0(E\otimes \xi)$, then
for nonzero $\alpha\in H^1(\mathscr O_C)$,
$[\alpha]\subseteq C_\xi\Theta$ if and only if
 the cup product map
$\cup \alpha :H^0(E\otimes \xi)\to H^1(E\otimes \xi)$ has a nontrivial kernel.  Now let $p\in C$, and consider the exact sequence
$$
0\to \mathscr O_C \to \mathscr O_C(p)\to \mathscr O_p(p) \to 0.
$$
This induces a coboundary map $\delta:H^0(\mathscr O_p(p))\to H^1(\mathscr O_C)$, and after identifying
$H^1(\mathscr O_C)=H^0(K_C)^\vee$, one can check that for nonzero $t\in H^0(\mathscr O_p(p))$,
$[\delta(t)]=\phi_K(p)\in \mathbb PH^0(K_C)^\vee$.
On the other hand, after tensoring the above exact sequence by $E\otimes \xi$, there is an induced coboundary map
$\partial:H^0(E\otimes \xi(p)|_p)\to H^1(E\otimes \xi)$, and Lemma \ref{lemcomp} shows that  $\cup \delta(t):H^0(E\otimes \xi)\to H^1(E\otimes \xi)$ is given by the composition
$$
H^0(E\otimes \xi)\stackrel{\cup t}{\to}H^0(E\otimes \xi(p)|_p)\stackrel{
\partial}{\to}H^1(E\otimes \xi).
$$

Thus if $h^0(E\otimes \xi(-p))\ne 0$, then $p\in \phi_K(C)$ is contained in $\mathbb P C_L\Theta$.  This is the case if $h^0(E\otimes \xi)>r$.  The proof of the general case is similar, replacing $p$ with a general effective divisor of degree $n+1$.
\end{proof}

\begin{cor}\label{corsecant}
Suppose $E\in \mathscr U(r,r(g-1))$ is general, and $\xi\in \textnormal{Pic}^0(C)$.  If $h^0(E\otimes \xi)>nr$, then
$s_{n-1}(\phi_K(C))\subseteq \mathbb{P}C_\xi \Theta_E$.
\end{cor}

\begin{rem}
In the well known case that $n=r=1$, this is saying that the canonical curve is contained in every tangent cone to a singular point of the theta divisor of the Jacobian.
It is a result of Green \cite{green} that in fact the quadric tangent cones cut out the canonical curve as a scheme.    For higher rank, it would be interesting to know if the canonical curve was cut out by degree $r+1$ tangent cones to $r$-theta divisors in the Jacobian.
\end{rem}



\begin{thebibliography}{99}


\bibitem{acgh} E. Arbarello, M. Cornalba, P.A. Griffiths and J. Harris,  \emph{Geometry of algebraic curves. Vol. I}. Grundlehren der Mathematischen Wissenschaften [Fundamental Principles of Mathematical Sciences], 267. Springer-Verlag, New York, 1985.

 \bibitem{n}  S. B. Bradlow, O. Garc\'ia-Prada, V. Mu\~n–oz and P.E. Newstead,  \emph{Coherent systems and Brill-Noether theory}. Internat. J. Math. 14 (2003), no. 7, 683--733.

\bibitem{dask}
S. B. Bradlow, G. Daskalopoulos, O. Garc\'ia-Prada, R. Wentworth, \emph{Stable augmented bundles over Riemann surfaces}. Vector bundles in algebraic geometry (Durham, 1993), 15--67, 
London Math. Soc. Lecture Note Ser., 208, Cambridge Univ. Press, Cambridge, 1995. 


\bibitem{cmf} S. Casalaina-Martin and R. Friedman, \emph{Cubic threefolds and abelian varieties of dimension five}. J. Algebraic Geom. 14 (2005), no. 2, 295--326.

\bibitem{c} S. Casalaina-Martin, \emph{Singularities of the Prym theta divisor}.  Ann. of Math. 170 (2009), 163--204.

\bibitem{cc} S. Casalaina-Martin, \emph{Cubic threefolds and abelian varieties of dimension five. II}.  Math. Z. 260 (2008), no. 1, 115--125.

\bibitem{cs}
J. Cilleruelo and I. Sols, \emph{The Severi bound on sections of rank two semistable bundles on a Riemann surface}. Ann. of Math. (2) 154 (2001), no. 3, 739--758.

\bibitem{el} L. Ein and R. Lazarsfeld, \emph{Singularities of theta divisors and the birational geometry of irregular varieties}.
J. Amer. Math. Soc. 10 (1997), no. 1, 243-258.

\bibitem{eh} D. Eisenbud and J. Harris, \emph{Vector spaces of matrices of low rank}.
Adv. in Math. 70 (1988), no. 2, 135-155.


\bibitem{f} W. Fulton, \emph{Intersection theory. Second edition}. Ergebnisse der Mathematik und ihrer Grenzgebiete. 3. Folge. A Series of Modern Surveys in Mathematics, 2. Springer-Verlag, Berlin, 1998.


\bibitem{green} M. Green, \emph{Quadrics of rank four in the ideal of a canonical curve}. Invent. Math. 75, (1984), no. 1, 85--104,


\bibitem{gt} I. Grzegorczyk and  M. Teixidor i Bigas, \emph{Brill-Noether theory for stable vector
bundles}.
Moduli Spaces and Vector Bundles,  ed. S. Bradlow, L.
Brambila-Paz,  O. Garcia-Prada and S. Ramanan Editors, LMS
 Lecture Note Series 359, 29--50, Cambridge University Press, 2009.





\bibitem{hart} R. Hartshorne,  \emph{Algebraic geometry}. Graduate Texts in Mathematics, No. 52. Springer-Verlag, New York-Heidelberg, 1977.

\bibitem{he} M. He, \emph{Espaces de modules de syst\`emes coh\'erent}. Internat. J. Math. 9 (1998), no. 5, 545--598.


\bibitem{h} G. Hitching, \emph{Geometry of vector bundle extensions and applications to the generalised theta divisor}. arXiv math.AG/0610970.

\bibitem{kt} G. Kempf, \emph{On the geometry of a theorem of Riemann}. Ann. of Math. (2) 98 (1973), 178--185.

\bibitem{ka} G. Kempf,  \emph{Abelian Integrals}.
 Monograf\'ias del Instituto de Matem\'aticas, 13. Universidad Nacional Aut\'onoma de M\'exico, M\'exico, 1983. vii+225 pp.


\bibitem{kn} A. King and P.E. Newstead, \emph{Moduli of Brill-Noether pairs on algebraic curves}.
 Internat. J. Math. 6 (1995), no. 5, 733--748.


\bibitem{kol} J. Koll\'ar, \emph{Shafarevich Maps and Automorphic Forms}. Princeton University Press, 1995.

\bibitem{Li} Y.Li, \emph{Spectral curves, theta divisors and Picard
bundles}. Internat. J. Math. 2 (1991), no. 5, 525--550.



\bibitem{ln}
H. Lange and P.E. Newstead, \emph{On Clifford's theorem for rank-3 bundles}. Rev. Mat. Iberoamericana 22 (2006), no. 1, 287--304.

\bibitem{l} Y. Laszlo, \emph{Un th\'eor\`eme de Riemann pour les diviseurs th\^eta sur les espaces de modules de fibr\'es stables sur une courbe}.  Duke Math. J. 64 (1991), no. 2, 333--347.




\bibitem{laum} G. Laumon, \emph{Fibr\'es vectoriels sp\'eciaux}.  Bull. Soc. Math. France 119 (1991), no. 1,  97--119.

\bibitem{laz} R. Lazarsfeld, \emph{Positivity in algebraic geometry: I \& II}. Springer-Verlag, Berlin, 2004. xviii+385 pp.


\bibitem{lepot} J. Le Potier, \emph{Faisceaux semi-stables et syst\`emes coh\'erent}. Vector Bundles in Algebraic Geometry, Durham 1993, ed. N.J. Hitchin, P.E. Newstead and W.M. Oxbury, LMS Lecture Notes Series 208, 179--329, Cambridge University Press, 1995.

\bibitem{m} D. Mumford,  \emph{Abelian Varieties}.
 Tata Institute of Fundamental Research Studies in Mathematics, No. 5 Published for the Tata Institute of Fundamental Research, Bombay; Oxford University Press, London 1970 viii+242 pp.

\bibitem{nr} M. Narasimhan and S. Ramanan, \emph{Moduli of vector bundles on a compact Riemann surface}. Ann. of Math. (2) 89 (1969) 14--51.


\bibitem{ns} M. Narasimhan and C.S. Seshadri, \emph{Stable and unitari vector bundles on a compact Riemann surface}.
 Ann. of Math. (2) 82 (1965) 540--567.

\bibitem{ravi} N. Raghavendra and P.A. Vishwanath, \emph{Moduli of pairs and generalized theta divisors}. T\^ohuko Math. J. (2) 46 (1994), no. 3, 321--340.




 \bibitem{ramanan}  S. Ramanan, \emph{The theory of vector bundles on algebraic curves with some applicatons}. Moduli Spaces and Vector Bundles,  ed. S. Bradlow, L.
Brambila-Paz,  O. Garcia-Prada and S. Ramanan Editors, LMS
 Lecture Note Series 359, 165--209, Cambridge University Press, 2009.


 \bibitem{r} M. Raynaud, \emph{Sections des fibr\'es vectoriels sur une courbe}. Bull. Soc. Math. France 110 (1982), 103-135.

\bibitem{RT} B. Russo and  M. Teixidor i Bigas, \emph{On a conjecture of Lange}.
 J.of Alg. Geom 8 (1999), 483-496.

\bibitem{sundaram} N. Sundaram, \emph{Special divisors and vector bundles}. T\^ohuku Math. J. (2) 39 (1987), no. 2, 175--213.

\bibitem{tbn}  M. Teixidor i Bigas, \emph{Brill-Noether theory for stable vector bundles}. Duke Math. J. (62) 2   (1991), 385--400.

\bibitem{t}  M. Teixidor i Bigas, \emph{A Riemann singularity theorem for generalised Brill-Noether loci}. Math. Nachr. 196 (1998), 251--257.






\end{thebibliography}
\end{document}